\documentclass[11pt]{amsart}

\usepackage{latexsym}
\usepackage{amsmath}
\usepackage{amsfonts}
\usepackage{amssymb}
\usepackage{tikz,float}
\usepackage{cite}
\usepackage{enumitem}
\usepackage{listings}
\usepackage{blkarray}
\usepackage{subcaption}
\newtheorem{theorem}{Theorem}[section]
\newtheorem{lemma}[theorem]{Lemma}
\newtheorem{corollary}[theorem]{Corollary}

\newtheorem{sublemma}{}[theorem]

\theoremstyle{definition}

\theoremstyle{remark}

\numberwithin{equation}{section}

\DeclareMathOperator{\cl}{cl}
\DeclareMathOperator{\si}{si}
\DeclareMathOperator{\co}{co}

\usepackage{tikz, comment,xcolor,nicematrix,amsmath,graphicx, pifont,caption,picture,multirow, array}
\usetikzlibrary {quotes} 
\usepackage{stmaryrd,multimedia}
\usetikzlibrary{matrix, positioning, arrows.meta, decorations.pathreplacing,decorations.markings, decorations.shapes,backgrounds,overlay-beamer-styles,quotes}

\begin{document}

\title[Matroids With Minimum Cocircuit Size four]{Unavoidable Minors of Matroids with Minimum Cocircuit Size four}

\author{Matthew Mizell}
\address{Mathematics Department \\
Louisiana State University \\ 
Baton Rouge, Louisiana}
\email{mmizel4@lsu.edu}

\author{James Oxley}
\address{Mathematics Department \\
Louisiana State University \\ Baton Rouge, Louisiana}
\email{oxley@math.lsu.edu}

\date{\today}

\begin{abstract}
In 1963, Halin and Jung proved that every simple graph with minimum degree at least four has $K_5$ or $K_{2,2,2}$ as a minor. Mills and Turner proved an analog of this theorem by showing that every $3$-connected binary matroid in which every cocircuit has size at least four has $F_7, M^*(K_{3,3}), M(K_5),$ or $ M(K_{2,2,2})$  as a minor. Generalizing these results, this paper proves that every simple matroid in which all cocircuits have at least four elements has as a minor one of nine matroids, seven of which are well known. All nine of these special matroids have rank at most five and have at most twelve elements.
\end{abstract}
%%%%%%%%%%%%%%%%%%%%%%%%%%%%%%%%%%%%%%%%%%%%%%%
\maketitle
\section{Introduction}

The purpose of this paper is to prove a matroid analog of the following result of Halin and Jung [\ref{HJ}]. 
 \begin{theorem}\label{mainTheorem:graphic}
    Every simple graph with minimum degree at least four has $K_5$ or $K_{2,2,2}$ as a minor. 
\end{theorem}

We have followed Bollob\'{a}s [\ref{boll}, p.$373$] and Maharry [\ref{MAH}, p.96] in attributing Theorem \ref{mainTheorem:graphic} to Halin and Jung. Fijavž and Wood [\ref{FijWood}, Corollary A.4] give a short proof of that theorem and briefly discuss its origins.

When $G$ is a $2$-connected loopless graph, the set of edges that meet a fixed vertex of $G$ is a bond of $G$ and a cocircuit of its cycle matroid $M(G)$. Because of this, it is common in matroid theory to take minimum cocircuit size as a matroid analog of minimum vertex degree in a graph. Moreover, the minimum cocircuit size $M(G)$ is precisely the edge connectivity of $G$.

 The next theorem is the main result of the paper. Most of the matroids appearing in it are familiar. Geometric representations of the rank-$3$ matroids $P_7$ and $O_7$ are shown in Figure \ref{fig:test}. We define $H_{12}$ to be the $12$-element rank-$5$ matroid $O_7 \oplus_2 O_7$ where the basepoint of the $2$-sum in each copy of $O_7$ is the point $p$ denoted in Figure \ref{fig:test}. We recall that $M^*(K_{3,3})$, the bond matroid of $K_{3,3}$, is the rank-$4$ matroid that can be obtained from a twisted $3\times3$ grid (see [\ref{James}, p.652]). The matroid $Q_9$, for which a geometric representation is shown in Figure \ref{fig:K}, is a $9$-element rank-$4$ matroid that is obtained from a twisted $3\times4$ grid by removing the three boxed elements as shown. It is straightforward to check that this matroid is representable over a field $\mathbb{F}$ if and only if $|\mathbb{F}| \geq 3$. The matroid $Q_9$ is represented by $[I_4|A]$ for the ternary matrix $A$ in Figure \ref{fig:K}. One can also show that this matroid is affine over $GF(3)$.

\begin{theorem}\label{mainTheorem}
Every simple matroid in which every cocircuit has at least four elements has $U_{2,5}, F_7,F_7^-,P_7, M^*(K_{3,3}), Q_9, M(K_5), M(K_{2,2,2})$, or $H_{12}$ as a minor.
\end{theorem}

It is straightforward to check that each of the nine matroids listed in this theorem is a minor-minimal simple matroid in which every cocircuit has size at least four. As consequences of this theorem, we have the next two results. The first of these was proved by Mills and Turner [\ref{MILTUR}]. The second is the key step in the proof of Theorem \ref{mainTheorem} and its proof occupies most of the paper.
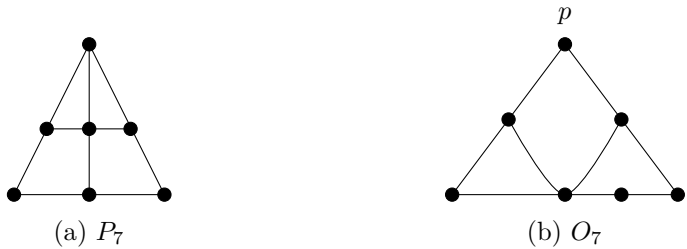
\begin{figure}
\centering
\begin{subfigure}{.5\textwidth}
  \centering
  \begin{tikzpicture}[scale=0.50]

            	\coordinate (a) at (0,0); 
                \coordinate (b) at (2,0);
                \coordinate (c) at (4,0);
                \coordinate (d) at (2,4);
                \coordinate (e) at (.87,1.75);
                \coordinate (f) at (2.,1.75);
                \coordinate (g) at (3.1,1.75);
                
                \draw[fill=black] (a) circle (5pt);
                \draw[fill=black] (b) circle (5pt);
                \draw[fill=black] (c) circle (5pt);
                \draw[fill=black] (d) circle (5pt);
                \draw[fill=black] (e) circle (5pt);
                \draw[fill=black] (f) circle (5pt);
                \draw[fill=black] (g) circle (5pt);

                \draw (a) -- (d) -- (c) -- (a);
                \draw (d) -- (b);
                \draw (e) -- (g);

			\end{tikzpicture}
            \caption{$P_7$}
  \label{fig:sub1}
\end{subfigure}%
\begin{subfigure}{.5\textwidth}
  \centering

  \begin{tikzpicture}[scale=0.50]

            	\coordinate (a) at (0,0); 
                \coordinate (b) at (3,0);
                \coordinate (c) at (4.5,0);

                \coordinate (d) at (6,0);
                \coordinate (e) at (3,4);

                \coordinate (f) at (1.5,2);

                \coordinate (g) at (4.5,2);
                \draw[fill=black] (a) circle (5pt);
                \draw[fill=black] (b) circle (5pt);
                \draw[fill=black] (c) circle (5pt);
                \draw[fill=black] (d) circle (5pt);

                \draw[fill=black] (e) circle (5pt) [label=above:$p$]{};

                \node at (3,4)[label=above:\small $p$]{};

                \draw[fill=black] (f) circle (5pt);

                \draw[fill=black] (g) circle (5pt);

                \draw (a) -- (d);

                \draw (a) --(e) -- (d);

                \draw  plot [smooth] coordinates {(f) (b) (g)};
			\end{tikzpicture}
  
  \caption{$O_7$}
  \label{fig:sub2}
\end{subfigure}
\caption{The matroids $P_7$ and $O_7$. }
\label{fig:test}
\end{figure}

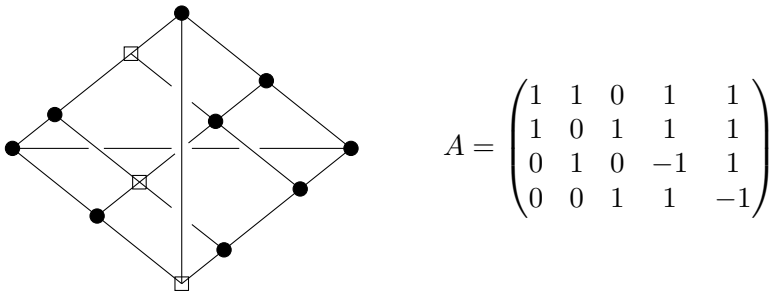
\begin{figure}[b]
    \centering
    \begin{minipage}{.5\textwidth}
        \centering
        \begin{tikzpicture}[scale=0.45]

				\coordinate (t1) at (2.5,2);
				\coordinate (left) at (0,4);

                \coordinate (leftup1) at (1.25,5);
                \coordinate (top) at (5,8);

                \coordinate (bottomright1) at (6.25,1);
                \coordinate (bottomright2) at (8.5, 2.8);
                \coordinate (right) at (10,4);

                \coordinate (topright) at (7.5,6);
                \coordinate (middletop) at (6,4.8);
                
                \coordinate (hole1) at (5,0); 
                \coordinate (hole1BR) at (5.2,-0.2);
                \coordinate (hole1TR) at (5.2, 0.2);
                \coordinate (hole1TL) at (4.8, 0.2);
                \coordinate (hole1BL) at (4.8, -0.2);

                \draw (hole1BR) -- (hole1TR) -- (hole1TL) -- (hole1BL) -- (hole1BR);
                
                \coordinate (hole2) at (3.5, 6.8);
                \coordinate (hole2BR) at (3.7,6.6);
                \coordinate (hole2TR) at (3.7,7);
                \coordinate (hole2TL) at (3.3,7);
                \coordinate (hole2BL) at (3.3,6.6);
                
                \draw (hole2BR) -- (hole2TR) -- (hole2TL) -- (hole2BL) -- (hole2BR);
                
                \coordinate (middlehole) at (3.75,3);
                \coordinate (MholeBR) at (3.95,2.8);
                \coordinate (MholeTR) at (3.95,3.2);
                \coordinate (MholeTL) at (3.55,3.2);
                \coordinate (MholeBL) at (3.55,2.8);

                \draw (MholeBR) -- (MholeTR) -- (MholeTL) -- (MholeBL) -- (MholeBR);
                
                \coordinate (gap1) at (4.7,2.25);
                \coordinate (gap2) at (5.3,1.75);

                \coordinate (gap3) at (4.7, 5.84);
                \coordinate (gap4) at (5.3, 5.36);

                \coordinate (LR1) at (2.25,4);
                \coordinate (LR2) at (2.7,4);
                \coordinate (LR3) at (4.7,4);
                \coordinate (LR4) at (5.3,4);
                \coordinate (LR5) at (6.7,4);
                \coordinate (LR6) at (7.3,4);

                \draw (hole1) -- (top);
                
                %\draw[fill=black] (gap4) circle (5pt);

                \draw (left) -- (LR1);
                \draw (LR2) -- (LR3);
                \draw (LR4) -- (LR5);
                \draw (LR6) -- (right);

                \draw (leftup1) -- (gap1);
                \draw (gap2) -- (bottomright1);
                \draw (hole2) -- (gap3);
                \draw (gap4) -- (bottomright2);
                
               % \draw (updownGap3) -- (hole1);
                
                \draw[fill=black] (t1) circle (6pt);
                \draw[fill=black] (left) circle (6pt);
                %\draw[fill=black] (hole1) circle (5pt);

                \draw[fill=black] (leftup1) circle (6pt);
                %\draw[fill=black] (hole2) circle (5pt);
                \draw[fill=black] (right) circle (6pt);
                \draw[fill=black] (top) circle (6pt);
                \draw[fill=black] (bottomright1) circle (6pt);
               % \draw[fill=black] (middlehole) circle (5pt);
                
                \draw[fill=black] (topright) circle (6pt);
                \draw[fill=black] (bottomright2) circle (6pt);

                \draw[fill=black] (middletop) circle (6pt);
                %/$draw[fill=black] (t1box) circle (5pt);
                %draw[fill=black] (t3box) circle (5pt);

                 \draw (hole1) -- (right);
                 \draw (hole1) -- (left);
                 \draw (top) -- (right);
                 \draw (top) -- (left);
                 %\draw (topright) -- (t1);
                 \draw (t1) -- (4.8,3.84);
                 \draw (5.2,4.16) -- (topright);
                 %\draw (bottomright1) -- (leftup1);
                 %\draw (hole2) -- (bottomright2);
			\end{tikzpicture}
    \end{minipage}%
    \begin{minipage}{0.4\textwidth}
        \centering
         $A =  \begin{pmatrix}
1 & 1 & 0 & 1 & 1\\
  1 & 0 & 1 & 1 & 1\\
  0 & 1 & 0 & -1 & 1\\
  0 & 0 & 1 & 1 & -1 \\
\end{pmatrix} $  
    \end{minipage}

        \caption{The rank-$4$ ternary affine matroid $Q_9$ and the matrix $A$, where $[I_4|A]$ is a ternary representation of $Q_9$.}
    \label{fig:K}
\end{figure}
\begin{theorem}\label{MainTheorem:Binary}
    Let $M$ be a simple binary matroid in which every cocircuit has at least four elements. Then $M$ has $F_7, M^*(K_{3,3}), M(K_5)$, or $M(K_{2,2,2})$ as a minor.
\end{theorem}

  \begin{theorem}\label{maintheorem:ternary}
 
    Let $M$ be a simple ternary matroid in which every cocircuit has at least four elements. Then $M$ has $F_7^-, P_7,M^*(K_{3,3}), Q_9,M(K_5), \linebreak M(K_{2,2,2})$, or $H_{12}$ as a minor.
    \end{theorem}

Recall that $K_{2,2,2}$ is the octahedron, so its planar dual is the cube. Using this, we see that the next result is the dual of Theorem~\ref{mainTheorem}.

\begin{corollary}
    Every matroid in which all cocircuits have at least three elements and all circuits have at least four elements has, as a minor, $U_{3,5}, F_7^*,\linebreak (F_7^-)^*, P_7^*, Q_9^*, M^*(K_5), H_{12}^*$, or the cycle matroid of the cube or $K_{3,3}$.
\end{corollary}

Applying this result to graphs, we immediately obtain the following well-known result. 

\begin{corollary}
    Every $3$-edge-connected graph with girth at least four has the cube or $K_{3,3}$ as a minor.
\end{corollary}

The study of matroids with many small circuits and cocircuits started with Tutte [\ref{tutteWheels}] when in his Wheels-and-Whirls Theorem, he proved that the only $3$-connected matroids in which every element is in a $3$-circuit and a $3$-cocircuit are wheels and whirls. Miller [\ref{Miller}] found all the matroids with at least thirteen elements such that every pair of elements is in a $4$-circuit and a $4$-cocircuit. Motivated by these results, Pfeil, Oxley, Semple, and Whittle~[\ref{pfiel}] found the $3$-connected matroids with the property that every pair of elements is in a $4$-circuit and every element is in a $3$-cocircuit.  The nine matroids listed in Theorem \ref{mainTheorem} have the property that every element is in a $3$-circuit and in a $4$-cocircuit.

\section{Preliminaries}

Throughout this paper, we will follow the notation and terminology of [\ref{James}]. We denote by $\mathcal{M}_4$ the class of simple matroids in which every cocircuit has at least four elements. The \textit{connectivity function} $\lambda_M$ of a matroid $M$ is defined, for all subsets $X$ of $E(M)$, by
 \begin{align*}
     \lambda_M(X) = r(X) + r^*(X) - |X|.
 \end{align*}
 When it is clear which matroid we are referring to, we will use $\lambda(X)$ in place of $\lambda_M(X)
 $. For disjoint subsets $X$ and $Y$ of $E(M)$, let $\kappa_M(X,Y) = \min\{\lambda_M(S) : X \subseteq S \subseteq E(M) - Y\}$. If $S$ is a set for which this minimum is attained, then $\kappa_M(X,Y) = \lambda_M(S) = \kappa_M(S,E(M)-S)$. In many of our proofs we will use Geelen, Gerards, and Whittle's extension [\ref{ggw}] (see also, for example, [\ref{James}, Theorem 8.5.7]) of Tutte's Linking Theorem [\ref{tutte1}].

\begin{theorem}\label{GGWTutte}
    Let $X$ and $Y$ be disjoint subsets of the ground set of a matroid $M$. Then $M$ has a minor $N$ with  $E(N) = X\cup Y$ for which $\kappa_N(X,Y) = \kappa_M(X,Y)$ such that $N|X = M|X$ and $N|Y = M|Y$.
\end{theorem}

The following result of D. W. Hall [\ref{DWHall}] will be used in the proof of Theorem~\ref{MainTheorem:Binary}. 

\begin{theorem}\label{HallTheorem}
    If $G$ is a $3$-connected graph, then $G$ has no $K_{3,3}$-minor if and only if $G$ is planar or its associated simple graph is $K_5$.
\end{theorem}

The next theorem for ternary matroids is reminiscent of the last result. Hall, Mayhew, and van Zwam [\ref{superfluous}] considered similar such results.

\begin{theorem}\label{ternarySpliiter}
    Let $M$ be a $3$-connected matroid having rank and corank at least three. Then $M$ is ternary if and only if $M$ has no $U_{2,5}$- or $F_7$-minor and $M \not\cong F_7^*$. 
\end{theorem}

\begin{proof}
  If $M$ is ternary, then $M$ has no $U_{2,5}$- or $F_7$-minor and $M \not\cong F_7^*$. Conversely, assume $M$ has no $U_{2,5}$- or $F_7$-minor and $M \not\cong F_7^*$. As $M$ is $3$-connected having rank and corank at least three, $M$ does not have $U_{3,5}$ as a minor by [\ref{oxl89}] (see also [\ref{James}, Proposition 12.2.5]). Then by [\ref{bixby0}, \ref{sey0}] (see also [\ref{James}, Theorem 6.5.7]), $M$ is ternary unless $M$ has $F_7^*$ as a minor. Since $M \not\cong  F_7^*$, by the Splitter Theorem, $M$ has a $3$-connected single-element extension or coextension of $F_7^*$ as a minor. Now, a $3$-connected extension $N$ of $F_7$ by the element $e$ either adds $e$ freely to a line of $F_7$, or adds $e$ freely to $F_7$ itself. In each case, $N/e$ has a $U_{2,5}$-minor, so $N^*$ has a $U_{3,5}$-minor, a contradiction.

\begin{figure}
    \centering
  \begin{tikzpicture}[tight/.style={inner sep=.5pt}, loose/.style={inner sep=.7em}, scale=0.60]

    \coordinate (edgeb) at (0,0);
    \coordinate (edget) at (0,6);
    
    \coordinate (rightpagetr) at (5,4.5);
    \coordinate (rightpagebr) at (5,-1.5);
    \coordinate (leftpagetl) at (-5,4.5);
    \coordinate (leftpagebr) at (-5,-1.5);

    \draw (edgeb) -- (edget) -- (rightpagetr) -- (rightpagebr) -- (edgeb);
    \draw (edgeb) -- (edget) -- (leftpagetl) -- (leftpagebr) -- (edgeb);

    \coordinate (bb) at (0,1.5);
    \draw (-.15,1.35) rectangle (.15,1.65);
    \node at (bb)[label= right:\small $t$, tight]{};
    
    \coordinate (bm) at (0,3.5);
    \draw (-.15,3.35) rectangle (.15,3.65);
    \node at (bm)[label= left:\small $s$, tight]{};   
    
    \coordinate (bt) at (0,5);
    \draw (-.15,4.85) rectangle (.15,5.15);
    \node at (bt)[label= right:\small $c$, tight]{};   

    %\draw[fill=black] (bb) circle (5pt);
    %\draw[fill=black] (bm) circle (5pt);
    %\draw[fill=black] (bt) circle (5pt);

    \coordinate (a) at (-4.2,.25);
    \node at (a)[label= below:\small $a$,tight]{};     
    \coordinate (u) at (-2.1,.85);
    \node at (u)[label= below:\small $u$, tight]{};     
    \coordinate (b) at (-1.8,3);
    \node at (b)[label= left:\small $b$, tight]{};     
    \coordinate (v) at (-1.23, 2.57);
    \node at (v)[label= below:\small $v$, tight]{};   
    \coordinate (help) at (-.8, 2);

    \draw[fill=black] (a) circle (5pt);
    \draw[fill=black] (u) circle (5pt);
    \draw[fill=black] (b) circle (5pt);
    \draw[fill=black] (v) circle (5pt);
    %\draw[fill=black] (help) circle (5pt);

    \draw (a) -- (bb);
    \draw (a) -- (bt);
    \draw (a) -- (bm);
    \draw (u) -- (bt);
    \draw (b) -- (bb);
    \draw  plot [smooth] coordinates {(b) (u) (help) (bm)};

    \coordinate (h) at (3.7,2);
    \coordinate (f) at (1.8, 3.55);
    \coordinate (g) at (1.3, 2.95);

    \draw[fill=black] (h) circle (5pt);
    \node at (h)[label= below:\small $h$, tight]{};
    \draw[fill=black] (f) circle (5pt);
    \node at (f)[label= above:\small $f$, tight]{};
    \draw[fill=black] (g) circle (5pt);
    \node at (g)[label= below:\small $g$, tight]{};

    \draw (h) -- (bt);
    \draw (h) -- (bm);
    \draw (f) -- (bb);
  \end{tikzpicture}
    \caption{The labels of $F_7^*$ used in Theorem \ref{ternarySpliiter}.}
    \label{fig:F7*}
\end{figure}
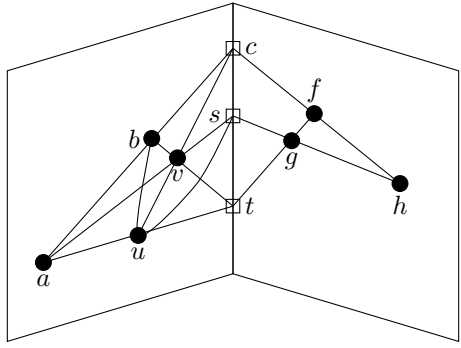

  We now need only consider the $3$-connected single-element extensions of $F_7^*$. A geometric representation of $F_7^*$ is shown in Figure \ref{fig:F7*} where $c,s,$ and $t$ are not elements of $F_7^*$, but show how $F_7^*$ can be obtained from $F_7$ by a $\Delta Y$-exchange. Let $N$ be a $3$-connected single-element extension of $F_7^*$ by the element $e$ and let $\mathcal{M}$ be the corresponding modular cut. Then $\mathcal{M}$ does not contain any rank-one flats. If $\mathcal{M} = \{E(F_7^*)\}$, then $e$ is freely added to $F_7^*$, and $N/e$ has $U_{3,5}$ as a minor, a contradiction. Thus $\mathcal{M}$ must contain some line or some hyperplane of $F_7^*$. Assume that $\mathcal{M}$ contains a line of $F_7^*$. As $F_7^*$ has a doubly transitive automorphism group [\ref{James}, p.643], we may assume that $\{a,b\} \in \mathcal{M}$. Then $\{a,b,u,v\}$ and $\{a,b,f,h\}$ are in $\mathcal{M}$. Assume $\{f,h\} \in \mathcal{M}$. Then $\{f,h,u,v\} \in \mathcal{M}$. As $\{a,b,u,v\}$ and $\{f,h,u,v\}$ form a modular pair, $\{u,v\} \in \mathcal{M}$. Similarly, if $\{u,v\} \in \mathcal{M}$, then $\{f,h\} \in \mathcal{M}$. We deduce that if we have $\{a,b\} \in \mathcal{M}$ and at least one of $\{f,h\}$ and $\{u,v\}$ is in $\mathcal{M}$, then the point we added corresponding to $\mathcal{M}$ is $c$ and the extension is isomorphic to $S_8$.

  If $e$ is added freely on $\{a,b\}$, then contracting $e$ and $g$ gives $U_{2,5}$ as a minor, a contradiction. We may now assume that $e$ is not added to any $2$-point line of $F_7^*$. We now know that the smallest flat in $\mathcal{M}$ has rank three. Because the hyperplanes of $F_7^*$ are of two types, complements of triangles in $F_7$ and complements of $4$-circuits in $F_7$, by symmetry, we may assume that $e$ is added on one of the planes spanned by $\{f,g,h\}$ or $\{a,b,u,v\}$.

  If $e$ is placed on the plane $\{a,b,u,v\}$, then, since it is not on any $2$-point lines, we see that $N|\{a,b,u,v,e\} \cong U_{3,5}$, a contradiction. We deduce that $e$ is placed on the plane spanned by $\{f,g,h\}$ but not on any of the lines spanned by $\{f,g\}, \{f,h\},$ or $\{g,h\}$. We know that $\{a,b,u,v\} \not\in \mathcal{M}$. Moreover, if $M$ is the principal modular cut generated by $\{f,g,h\}$, then $N/a/e$ has $U_{2,5}$ as a restriction, a contradiction.
\begin{figure}
    \centering
  \begin{tikzpicture}[scale=0.55]
    \coordinate (u) at (0,0);
    \node at (u)[label= left:\small $u$]{};
    \coordinate (h) at (6,0);
    \node at (h)[label= right:\small $h$]{};
    \coordinate (b) at (3,6);
    \node at (b)[label= left:\small $b$]{};
    \coordinate (v) at (1.5,3);
    \node at (v)[label= left:\small $v$]{};
    \coordinate (f) at (4.5,3);
    \node at (f)[label= right:\small $f$]{};
    \coordinate (g) at (3,2);
    \node at (g)[label= below:\small $g$]{};
    
    \draw[fill=black] (u) circle (5pt);
    \draw[fill=black] (h) circle (5pt);
    \draw[fill=black] (b) circle (5pt);
    \draw[fill=black] (v) circle (5pt);
    \draw[fill=black] (f) circle (5pt);
    \draw[fill=black] (g) circle (5pt);

    \draw (u) -- (b) -- (h);
    \draw (h) -- (v);
    \draw (u) -- (f);
  \end{tikzpicture}
    \caption{The matroid $N\backslash e/a$ in Theorem \ref{ternarySpliiter}.}
    \label{fig:mk4}
\end{figure}
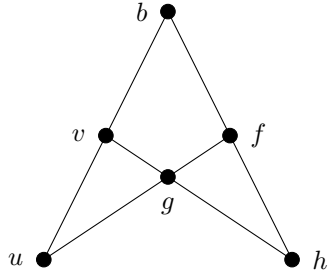
  We deduce that $\mathcal{M}$ contains a flat other than $\{f,g,h\}$ and $E(F_7^*)$. Now $N\backslash e / a$ is isomorphic to the copy of $M(K_4)$ labeled as in Figure \ref{fig:mk4}. By symmetry, we may assume that $\{a,b,e,g\}$ is a circuit of $N$. We see that $\{e,f,v\}$ and $\{e,h,u\}$ are not both circuits of $N / a$ otherwise $N/a \cong F_7$. By symmetry, we may assume that $\{e,h,u\}$ is not a circuit of $N /a$. If $\{e,f,v\}$ is not a circuit of $N/a$, then contracting $e$ from $N/a$ gives a matroid with $U_{2,5}$ as a restriction. Thus we may assume that $\{e,f,v\}$ is a circuit of $N/a$, so $\{a,e,f,v\}$ is a circuit of $N$. 

  We now know that $N$ has $\{a,b,e,g\}$ and $\{a,e,f,v\}$ as circuits. The matroid $N/h$ has $e$ on the line spanned by $\{f,g\}$. But $\{a,e,u\}$ is not a triangle of $N/h$. If $\{b,e,v\}$ is not a triangle of $N/ h$, then $(N/h)|\{a,b,e,u,v\} \cong U_{3,5}$. Thus $\{b,e,v\}$ is a triangle of $N/ h$. Hence $\{b,e,h,v\}$ is a circuit of $N$. 

  Consider $N/v$. We know that $N/v\backslash e \cong M(K_4)$. Also $N/v$ has $\{a,e,f\}$ and $\{b,e,h\}$ as triangles. Since $N/v$ is not isomorphic to $F_7$, we deduce that $\{e,g,u,v\}$ is not a circuit of $N$. Now consider $N/u$. We know that $\{a,e,h\}$ and $\{e,g,v\}$ are not triangles of this matroid. Then $N/u/e$ has $U_{2,5}$ as a restriction, a contradiction.
  \end{proof}

The next lemma identifies a key property of the minor-minimal members of $\mathcal{M}_4$ that will be used repeatedly throughout the paper.

\begin{lemma}\label{minCond}
    Let $M$ be a minor-minimal matroid in $\mathcal{M}_4$. Let $e$ be an element of $M$. Then $e$ is in a triangle and a $4$-cocircuit.
\end{lemma}
\begin{proof}
    Assume $e$ is not in a triangle. Consider $M/e$. Then every cocircuit in $M/e$ has size at least four and is simple. Thus $M/e$ contradicts the minimality of $M$. Therefore, $e$ is in a triangle.
    Now assume that $e$ is not in a $4$-cocircuit. Since $M \backslash e$ has a cocircuit $C^*$ of size less than four, we see that $C^* \cup e$ is a cocircuit of $M$ having size four . 
\end{proof}

\begin{lemma}\label{2Sum2Mat}
    Let $M$ be a minor-minimal matroid in $\mathcal{M}_4$. If $M$ is not $3$-connected, then $M= M_1 \oplus_2 M_2$ where $M_1$ and $M_2$ are $3$-connected and each has rank at least three.
\end{lemma}

\begin{proof}
    The minimality of $M$ implies that $M$ is $2$-connected. Let $T$ be the canonical tree decomposition of $M$ (see, for example, [\ref{James}, Section 8.3]). Consider a matroid $L$ that labels a leaf of $T$. If $L$ is a circuit, then $M$ has a $2$-cocircuit, a contradiction. Moreover, since $M$ is simple, $L$ is not a cocircuit. Thus $L$ is a $3$-connected matroid with at least four elements. Hence if $M_1$ and $M_2$ label distinct leaves of $T$, then $M$ has $M_1 \oplus _2 M_2$ as a minor. Therefore, as every cocircuit of $M_1\oplus_2M_2$ has size at least four, we deduce that $M = M_1 \oplus_2 M_2$. If $r(M_i) =2$ for some $i$, then, as every cocircuit of $M$ has at least four elements, we deduce that $|E(M_i)| \geq 5$, so $M$ has $U_{2,5}$ as a minor, a contradiction.
\end{proof}

Theorem \ref{MainTheorem:Binary} was originally proved by Mills and Turner [\ref{MILTUR}]; we provide their short proof for the sake of completeness. Note that the proof of Theorem~\ref{mainTheorem} does not rely on Theorem~\ref{MainTheorem:Binary}. Instead, Theorem~\ref{MainTheorem:Binary} can be deduced as an immediate corollary of Theorem~\ref{mainTheorem}.

\begin{proof}[Proof of Theorem \ref{MainTheorem:Binary}]
   Assume that $M$ has none of $F_7, M^*(K_{3,3}), M(K_5),$ or $ M(K_{2,2,2})$ as a minor. Then, as $M$ does not have an $F_7$-minor, it follows by [\ref{sey}] (see also [\ref{James}, Proposition 12.2.3]) that $M$ is regular otherwise $M \cong F_7^*$, which is a contradiction since $F_7^*$ has a triad. We show next that 

   \begin{sublemma}\label{binary:3-conn}
       $M$ is not $3$-connected. 
   \end{sublemma}

   Assume that $M$ is $3$-connected. Then, by a result of Seymour [\ref{sey}] (see also [\ref{James}, Theorem 13.1.2]), $M$ is graphic or cographic, or $M$ has a minor isomorphic to $R_{10}$ or $R_{12}$. By Theorem \ref{mainTheorem:graphic}, $M$ is not graphic. Suppose $M$ is cographic. Then, as $M$ is not graphic, $M$ is the bond matroid of a nonplanar graph $G$. Since $M$ does not have $M^*(K_{3,3})$ as a minor, it follows, by Theorem \ref{HallTheorem}, that $M \cong M^*(K_5)$. Thus $M$ has a triad, a contradiction. We conclude that $M$ is not cographic. Finally, if $M$ has $R_{10}$ or $R_{12}$ as a minor, then $M$ has $M^*(K_{3,3})$ as a minor, a contradiction. Thus \ref{binary:3-conn} holds.

   By Lemma~\ref{2Sum2Mat}, $M = M_1 \oplus_2 M_2$ where $M_1$ and $M_2$ are $3$-connected. As $M_1$ has none of $F_7, M^*(K_{3,3}), M(K_5),$ or $ M(K_{2,2,2})$ as a minor, \ref{binary:3-conn} implies that $M_1$ is not $3$-connected, a contradiction. 
\end{proof}

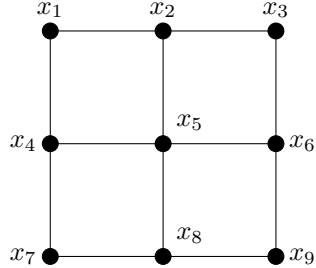
\begin{figure}
    \centering
    %\begin{subfigure}{.25\textwidth}
    
       \begin{tikzpicture}[tight/.style={inner sep=1pt}, loose/.style={inner sep=.7em}, scale=0.50]

			\coordinate (bl) at (0,0);
            \coordinate (bm) at (3,0);
            \coordinate (br) at (6,0);

            \coordinate (ml) at (0,3);
            \coordinate (mm) at (3,3);
            \coordinate (mr) at (6,3);

            \coordinate (tl) at (0,6);
            \coordinate (tm) at (3,6);
            \coordinate (tr) at (6,6);

            \draw[fill=black] (bl) circle (6pt)[label=left:]{};
            \node at (bl)[label=left:\small $x_7$, tight]{};
            
            \draw[fill=black] (bm) circle (6pt)[label=topright:$x_8$, tight]{};
            \node at (bm)[label=above right:\small $x_8$, tight]{};
            
            \draw[fill=black] (br) circle (6pt)[label=right:$x_9$, tight]{};
            \node at (br)[label= right:\small $x_9$, tight]{};

            \draw[fill=black] (ml) circle (6pt)[label=left:$x_4$, tight]{};
            \node at (ml)[label=left:\small $x_4$, tight]{};
            
            \draw[fill=black] (mm) circle (6pt)[label=topright:$x_5$, tight]{};
            \node at (mm)[label=above right:\small $x_5$, tight]{};
            
            \draw[fill=black] (mr) circle (6pt)[label=right:$x$]{};
            \node at (mr)[label=right:\small $x_6$, tight]{};
            
            \draw[fill=black] (tr) circle (6pt)[label=above:$x_3$]{};
            \node at (tr)[label=above:\small $x_3$, tight]{};
            \draw[fill=black] (tm) circle (6pt)[label=above:$x_2$]{};
            \node at (tm)[label=above:\small $x_2$, tight]{};
            \draw[fill=black] (tl) circle (6pt)[label=above:$x_1$]{};
            \node at (tl)[label=above:\small $x_1$, tight]{};
            
            \draw (bl) -- (br) -- (tr) -- (tl) -- (bl);      
            \draw (tm) -- (bm);
            \draw (ml) -- (mr);
			\end{tikzpicture}
 
%\end{subfigure}%
%\begin{subfigure}{.5\textwidth}
%  \centering
%       \[
%\begin{blockarray}{ccccccccc}
%1 & 2 & 3 & 4 & 5 & 6 & 7 & 8 & 9\\
%\begin{block}{(ccccccccc)}
%  1 & 0 & 0 & 0 & 1 & 1 & 1 & 2 & 0\\
%  0 & 1 & 0 & 0 & 1 & 0 & 1 & 1 & 0\\
%  0 & 0 & 1 & 0 & 0 & 1 & 1 & 0 & 1\\
%  0 & 0 & 0 & 1 & 1 & 0 & 0 & 1 & 1 \\
%\end{block}
%end{blockarray}
% \]   

%\end{subfigure}
    \caption{Some $3$-point lines in a rank-$4$ matroid.}
    \label{fig:k33}
\end{figure}

\begin{lemma}\label{lem33}
    Let $M$ be a $9$-element rank-$4$ matroid that has each of the $3$-point lines in Figure $\ref{fig:k33}$ as a triangle. Then $M$ is isomorphic to $M^*(K_{3,3})$.
\end{lemma}

\begin{proof}
We observe that, as $r(M) = 4$, every set of four points that form the vertices of a rectangle in Figure \ref{fig:k33} is a cocircuit of $M$. As $M$ has rank four, it has no other triangles apart from the six shown. Now, for all $i$ in $\{x_1,x_2,\dots,x_9\}$, the matroid $M/x_i$ is ternary since it can be obtained from $M(K_4)$ by adding parallel elements. Thus $M$ does not have $U_{2,5}, U_{3,5},$ or $F_7$ as a minor. Moreover, $M$ does not have $F_7^*$ as a minor since we cannot eliminate all of the triangles of $M$ by deleting two elements. We conclude that $M$ is ternary [\ref{bixby0}, \ref{sey0}].

Now $M$ has $\{x_1,x_2,x_4,x_5\}$ as a basis $B$ otherwise $r(M) = 3$. Let $[I_4|A_4]$ be a ternary representation of $M$ with respect to $B$. Scaling the rows and columns of $A_4$ so that the first non-zero entry of each is a one, we get that $A_4$ is as shown in Figure \ref{repnk33} by using the fundamental circuits with respect to $B$ along with the circuit $\{x_3,x_6,x_9\}$, where $u_1$ and $u_2$ are non-zero. Finally, by using the circuit $\{x_7,x_8,x_9\}$, we deduce that $u_2 = 1$ and $u_1 = 1$. Thus $M$ is represented by the matrix $[I_4|A_4]$ over $GF(3)$ where $A_4$ is as shown in Figure~\ref{repnk33} with $u_2 = 1 = u_1$. As $M^*(K_{3,3})$ is a rank-$4$ ternary matroid that has the six triangles indicated in the figure, we deduce that $M \cong M^*(K_{3,3})$.
\end{proof}

\begin{figure}
    \centering
\[
\begin{blockarray}{cccccc}
 & x_3 & x_7 & x_8 & x_6 & x_9 \\
\begin{block}{c(ccccc)}
  x_1     & 1 & 1 & 0 & 0   & 1    \\
  x_2     & 1 & 0 & 1 & 0   & 1    \\
  x_4     & 0 & 1 & 0 & 1   & u_2  \\
  x_5     & 0 & 0 & 1 & u_1 & u_1u_2\\
\end{block}
\end{blockarray}
 \]
 \caption{The matrix $A_4$ in the proof of Lemma~\ref{lem33}.}
   \label{repnk33}
\end{figure}
\begin{figure}
    \centering
    
  \begin{tikzpicture}[tight/.style={inner sep=1pt}, loose/.style={inner sep=.7em}, scale=0.55]

            	\coordinate (a) at (0,0); 
				\coordinate (b) at (0,4);
				\coordinate (c) at (0,6);
				\coordinate (d) at (2,6);
				\coordinate (e) at (6,6);
                \coordinate (f) at (6,2);
                \coordinate (g) at (6,0);
                \coordinate (h) at (2,4);
                \coordinate (i) at (4,4);
                \coordinate (j) at (2,2);
                \coordinate (k) at (4,2);
                \coordinate (l) at (4,0);
                
				\draw (a) -- (b) -- (c) -- (d) -- (e) -- (f) -- (g) -- (a);
                \draw (b) -- (h) -- (i) -- (k) -- (l);
                \draw (d) -- (h) -- (j) -- (f);

                \draw[fill=black] (a) circle (5pt);
                \node at (a)[label=below:\small $x_7$, tight]{};
                
                \draw[fill=black] (b) circle (5pt);
                \node at (b)[label=left:\small $x_{4}$, tight]{};
                
                \draw[fill=black] (c) circle (5pt);
                \node at (c)[label=above :\small $x_{1}$, tight]{};
                
                \draw[fill=black] (d) circle (5pt);
                \node at (d)[label=above:\small $x_{2}$, tight]{};
                
                \draw[fill=black] (e) circle (5pt);
                \node at (e)[label=above right:\small $x_3$, tight]{};
                
                \draw[fill=black] (f) circle (5pt);
                \node at (f)[label= right:\small $x_{11}$, tight]{};
                
                \draw[fill=black] (g) circle (5pt);
                \node at (g)[label=below:\small $x_{12}$, tight]{};
                
                \draw[fill=black] (h) circle (5pt);
                \node at (h)[label=above right:\small $x_5$, tight]{};
                
                \draw[fill=black] (i) circle (5pt);
                \node at (i)[label=above right:\small $x_6$, tight]{};
                \draw[fill=black] (j) circle (5pt);
                \node at (j)[label=above right:\small $x_8$, tight]{};
                \draw[fill=black] (k) circle (5pt);
                \node at (k)[label=above right:\small $x_9$, tight]{};
                \draw[fill=black] (l) circle (5pt);
                \node at (l)[label=below:\small $x_{10}$, tight]{};

  %\draw[rounded corners] (-1, -1) rectangle (4, 1);
			\end{tikzpicture}

    \caption{Some $3$-point lines in a rank-$5$ matroid.}
    \label{fig:k222}
\end{figure}
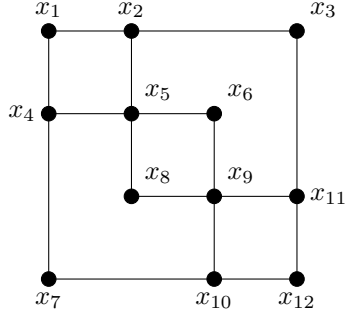
\begin{lemma}\label{tiedown}
    Let $M$ be a $12$-element rank-$5$ simple ternary matroid for which each of the $3$-point lines in Figure $\ref{fig:k222}$ is a triangle of $M$. Then $M$ is isomorphic to $M(K_{2,2,2})$.
\end{lemma}

\begin{proof}
    Since $r(M) = 5$, the triangles of $M$ imply that $\{x_1,x_2,x_4,x_5\},$ \linebreak    $\{x_1,x_3,x_7,x_{12}\}$, $\{x_2,x_3,x_8,x_{11}\},$ $\{x_4,x_6,x_7,x_{10}\}$,  $\{x_5,x_6,x_8,x_9\},$  and \linebreak $ \{x_9,x_{10},x_{11},x_{12}\}$  are cocircuits. Observe that these six sets coincide with the sets of corners of rectangles in Figure \ref{fig:k222}. Moreover each such set must be an independent set in $M$. We now construct a ternary representation $[I_5|A_5]$ for $M$. Let $\{x_1,x_2,x_4,x_5,x_9\}$ be the basis $B$ of $M$.

We shall scale the matrix $A_5$ so that the first non-zero entry of each column is a one. We also scale rows $2$--$5$ so that each has its first non-zero entry equal to one. The fundamental circuits of $x_3,x_7,x_8$, and $x_6$ with respect to $B$ imply that we may assume the first four columns are as shown where $u_1 \not= 0$. The cocircuits $\{x_1,x_3,x_7,x_{12}\},\{x_2,x_3,x_8,x_{11}\},$ and $\{x_4,x_6,x_7, x_{10}\}$ determine the first three rows of $x_{10}, x_{11},$ and $x_{12}$. The remaining two rows of $x_{10},x_{11},$ and $x_{12}$ are unknown. We label their entries as $u_2,u_3,\dots, u_6,$ and $u_7$ noting that these entries may be zero. The triangle $\{x_6,x_9,x_{10}\}$ implies that $u_1 = u_2$. The triangle $\{x_8,x_9,x_{11}\}$ implies that $u_3 = 1$. The triangle $\{x_7,x_{10},x_{12}\}$ implies that $u_4 = -u_2$ and $u_7 = -u_5$. The triangle $\{x_3,x_{11},x_{12}\}$ implies that $u_2 = 1$ and $u_6 = u_5$. Because the first non-zero entry of row $5$ is one, the matrix $A_5$ is as shown in Figure \ref{repk222-2}.

            \begin{figure}
    \centering
\[
\begin{blockarray}{cccccccccc}
 & x_3 & x_7 & x_8 & x_6 & x_{10}  & x_{11} & x_{12}\\
\begin{block}{c(ccccccccc)}
  x_1     & 1 & 1 & 0 & 0   & 0    &  0 &  1  \\
  x_2     & 1 & 0 & 1 & 0   & 0    &  1 &  0  \\
  x_4     & 0 & 1 & 0 & 1   & 1    &  0 & 0 \\
  x_5     & 0 & 0 & 1 & u_1 & u_2  & u_3& u_4 \\
  x_9     & 0 & 0 & 0 & 0   & u_5    &  u_6 &  u_7 \\
\end{block}
\end{blockarray}
 \]
 \caption{Building a ternary representation for $M$.}
    \label{repk222-1}
\end{figure}
\begin{figure}
    \centering
\[
\begin{blockarray}{cccccccccc}
 & x_3 & x_7 & x_8 & x_6 & x_{10}  & x_{11} & x_{12}\\
\begin{block}{c(ccccccccc)}
  x_1     & 1 & 1 & 0 & 0 & 0 & 0 & 1  \\
  x_2     & 1 & 0 & 1 & 0 & 0 & 1 & 0  \\
  x_4     & 0 & 1 & 0 & 1 & 1 & 0 & 0 \\
  x_5     & 0 & 0 & 1 & 1 & 1 & 1 & -1 \\
  x_9     & 0 & 0 & 0 & 0 & 1 & 1 & -1 \\
\end{block}
\end{blockarray}
 \]
 \caption{A ternary representation for $M$.}
    \label{repk222-2}
\end{figure}

Let $K_{2,2,2}$ be labeled as in Figure \ref{fig:graphk222}. We see that the eight triangles in this graph coincide with the eight $3$-point lines in Figure \ref{fig:k222}. Since $M(K_{2,2,2})$ is a simple ternary $12$-element rank-$5$ matroid and such a matroid with the specified eight triangles has the ternary representation $[I_5|A_5]$ where $A_5$ is as shown in Figure \ref{repk222-2}. We deduce that $M(K_{2,2,2}) \cong M[I_5|A_5]$, and the lemma is proved.
\end{proof}

\begin{figure}
    \centering
    \begin{tikzpicture}[tight/.style={inner sep=1pt}, loose/.style={inner sep=.7em}, scale = 0.4]
        \coordinate (bl) at (0,0);
        \coordinate (br) at (8,0);
        \coordinate (tr) at (8,8);
        \coordinate (tl) at (0,8);
        \coordinate (mid) at (4,4);
        \coordinate (top) at (4,12);

        \draw[fill=black] (bl) circle (6pt);
        \draw[fill=black] (tr) circle (6pt);
        \draw[fill=black] (br) circle (6pt);
        \draw[fill=black] (tl) circle (6pt);
        \draw[fill=black] (mid) circle (6pt);
        \draw[fill=black] (top) circle (6pt);
                \draw (bl) edge["$x_7$", tight]  (tl) 
                (tl) edge["$x_3$", tight] (tr)
                (tr) edge["$x_8$"] (br)
                (br) edge["$x_6$"] (bl)
                (mid) edge["$x_{10}$", tight] (bl)
                (mid) edge["$x_{11}$", tight] (tr)
                (br) edge["$x_9$", tight] (mid)
                (tl) edge["$x_{12}$", tight] (mid)
                (tr) edge["$x_2$", tight] (top)
                (top) edge["$x_1$", tight] (tl);

        \draw (top) to[bend left = 75, "$x_5$"] (br);
        \draw (bl) to[bend left = 75, "$x_4$"] (top);
        
    \end{tikzpicture}
    \caption{$K_{2,2,2}$}
    \label{fig:graphk222}
\end{figure}

\section{Structural Lemmas for Ternary Matroids}\label{Sec:Ternary}
This section contains the core of the proof of Theorem \ref{maintheorem:ternary}. Recall that $\mathcal{M}_4$ is the class of simple matroids in which every cocircuit has size at least four.

Because this section is long, we begin with an outline of the proof of Theorem \ref{maintheorem:ternary}. Let $M$ be a ternary minor-minimal matroid in $\mathcal{M}_4$ that is not isomorphic to $F_7^-,P_7,M^*(K_{3,3}),Q_9, M(K_5), M(K_{2,2,2}),$ or $H_{12}$. By Lemma~\ref{minCond}, every element of $M$ is in a triangle and a $4$-cocircuit. The proof strategy for Theorem \ref{maintheorem:ternary} involves analyzing how the many triangles and $4$-cocircuits of $M$ interact. We begin by considering what happens when $M$ has five of the six triangles in Figure \ref{fig:k33} as circuits. We then show, in Lemma~\ref{MainTheorem:2conn}, that $M$ must be $3$-connected. Lemma~\ref{No4PtLine} shows that $M$ has no $4$-point lines, and Lemmas~\ref{4CoCctNoTriangle} and \ref{4CocctInd} show that every $4$-cocircuit is independent. Next we show that $M$ cannot have two $4$-cocircuits contained in the union of two disjoint triangles. Lemmas~\ref{newton} and \ref{NoThreeTriangles} show that no element of $M$ is in more than two triangles. Lemmas~\ref{3tri} and \ref{ring} show that $M$ must have an element that is in more than one triangle. Lemmas~\ref{mrsub1}--\ref{new-1} identify and analyze an infinite family of matroids in $\mathcal{M}_4$ the first two members of which are $M(K_5)$ and $M(K_{2,2,2})$ and subsequent members of which have one of these two matroids as a minor. Lemmas \ref{no222},  \ref{new3}, \ref{new4}, and \ref{last} build from a particular $4$-cocircuit containing an element that is two triangles to get one of the forbidden possibilities for $M$. We now implement this strategy.

\begin{lemma}\label{newboy}
    Let $M$ be a rank-$4$ simple matroid having ground set $\linebreak \{x_1,x_2,\dots, x_9\}$. Suppose that $M$ has $\{x_1,x_2,x_3\}, \{x_4,x_5,x_6\},\ \{x_7,x_8,x_9\}, \linebreak\{x_1,x_4,x_7\}$, and  $\{x_3,x_6,x_9\}$ as triangles. Then $M \cong M^*(K_{3,3})$ or $M$ has $U_{2,5}$ or $P_7$ as a minor.
\end{lemma}

\begin{proof}
   Observe that if $\{x_2,x_5,x_8\}$ is a triangle of $M$ then, by Lemma~\ref{lem33}, $M \cong M^*(K_{3,3})$. Thus may assume that each of $x_2, x_5,$ and $x_8$ is in a unique triangle of $M$. Now consider $M / x_5 \backslash x_6$.  This has $\{x_1,x_2,x_3\}, \{x_7,x_8,x_9\}, \linebreak \{x_1,x_4,x_7\}$, and $\{x_3,x_4,x_9\}$ as triangles. We may assume that $\{x_2,x_4,x_8\}$ is not a triangle, otherwise this minor is isomorphic to $P_7$. It follows that $M/x_2,x_5\backslash x_3, x_6 \cong U_{2,5}$. 
\end{proof}

\begin{lemma}\label{MainTheorem:2conn}
    Let $M$ be a minor-minimal ternary matroid in $\mathcal{M}_4$. Then $M$ is $3$-connected or $M \cong H_{12}$.
\end{lemma}

\begin{proof}
    Assume that the result fails. By Lemma~$\ref{2Sum2Mat}$, $M = M_1 \oplus_2 M_2$, where $M_1$ and $M_2$ are $3$-connected matroids and $p$ is the basepoint of the $2$-sum. We may assume that $M_1$ or $M_2$, say $M_1$, is not isomorphic to $O_7$ having $p$ in its triad otherwise $M \cong H_{12}$. Observe that if $M_1$ has no triads, then the minimality of $M$ is contradicted. Thus every triad of $M_1$ contains $p$ otherwise $M$ has a triad. Let $\{x_1,x_3,p\}$ be a triad $T^*$ of $M_1$. Thus $M$ has a cocircuit $C^*$ that meets $E(M_1)$ in $\{x_1,x_3\}$. As Lemma~$\ref{minCond}$ implies that every element of $M$ is in a triangle, $M$ has a triangle containing $x_1$. By orthogonality with $C^*$, this triangle contains $x_3$ and an element $x_2$ of $E(M_1) - T^*$. Observe that $\{x_1,x_3\}$ is a cocircuit of $M_1 \backslash p$, so $\co(M_1 \backslash p)$ has a $2$-circuit. Thus $\co(M_1 \backslash p)$ is not $3$-connected, so, by Bixby's Lemma~[\ref{bixby}] (see also [\ref{James}, Lemma~8.7.3]), $\si(M_1/p)$ is $3$-connected. If neither $x_1$ nor $x_3$ is in a triangle of $M_1$ other than $\{x_1,x_2,x_3\}$, then $M_1 / p$ is $3$-connected and has no cocircuit of size less than four. This is a contradiction since $M_1 / p$ is a minor of $M$. Therefore, we may assume that $M_1$ has $\{p,x_1,x_4\}$ as a triangle for some element $x_4$. Then, either the only triangle containing $p$ in $M_1$ is $\{p,x_1,x_4\}$, or $M_1$ has another triangle containing $p$. By orthogonality, this triangle must be $\{p,x_3,x_5\}$ for some element $x_5$.

    We first assume that $\{p,x_3,x_5\}$ is a triangle of $M_1$. Since the intersection of $\cl_{M_1}(\{p,x_1,x_3\})$ and $E(M) - \{p,x_1,x_3\}$ contains $\{x_2,x_4,x_5\}$, this set is a triangle of $M$.  By Lemma~$\ref{minCond}$, $M$ has a $4$-cocircuit $C^*$ containing $x_2$. By symmetry and orthogonality, $x_4 \in C^*$. Furthermore, by orthogonality with a circuit  of $M$ that meets $E(M_1)$ in $\{x_1,x_4\}$, we deduce that $x_1 \in C^*$. Now suppose that $M_1$ has a point in $\cl_{M_1}(\{x_2,x_4\}) - \{x_2,x_4,x_5\}$. Then $M_1$ has an $O_7$-minor $M_1'$ having $p$ in the unique triad. Since $M_1' \oplus_2 M_2$ has no cocircuits of size less than four, we deduce that $M_1 \cong M_1' \cong O_7$, a contradiction. We may now assume that $C^* = \{x_1,x_2,x_4,x_6\}$ where $x_6 \not\in \cl_{M_1}(\{x_1,x_2,x_4\})$. By Lemma~$\ref{minCond}$, $x_5$ is in a $4$-cocircuit $D^*$ in $M$. By orthogonality and symmetry with $C^*$, we deduce that $x_3 \in D^*$. If $x_2 \not \in D^*$, then $x_4 \in D^*$ and $x_1 \in D^*$. Thus $D^* = \{x_1,x_3,x_4,x_5\}$. Then, because $\{p,x_1,x_2,\dots,x_5\}$ contains the cocircuits $D^*$ and $\{p,x_1,x_3\}$ of $M_1$, we deduce that $\lambda_{M_1}(\{p,x_1,x_2,x_3,x_4,x_5\}) \leq 3 + (6-2) - 6 = 1$. This is a contradiction as it implies that $x_6$ is a coloop of $M_1$. Therefore, we may assume that $x_2 \in D^*$ and neither $x_1$ nor $x_4$ is in $D^*$.  By Lemma~$\ref{minCond}$, $x_6$ is in a triangle of $M_1$, so $\{x_4,x_6,x_7\}$ or $\{x_2,x_6,x_7\}$ is a triangle of $M_1$ for some element $x_7$. 

    \begin{sublemma}\label{2conn:x_2x_6x_7notATriangle}
       When $\{p,x_3,x_5\}$ is a triangle, $M_1$ has no triangle containing $\{x_2,x_6\}$.
    \end{sublemma}
    Assume that $\{x_2,x_6,x_7\}$ is a triangle. By orthogonality between this triangle and $D^*$, we see that $D^* = \{x_2,x_3,x_5,x_6\}$ or $D^* = \{x_2,x_3,x_5,x_7\}$. Since $M_1$ has $D^*$, $C^*$, and $\{p,x_1,x_3\}$ as cocircuits, $$\lambda_{M_1}(\{p,x_1,x_2,\dots,x_7\}) \leq 4 + (8-3) - 8 = 1.$$ As $M$ is $3$-connected, $|E(M) - \{p,x_1,x_2,\dots,x_7\}| \leq 1$, so $M_1$ has a cocircuit containing $\{x_6,x_7\}$ that has size at most three and that does not contain $p$, a contradiction. Thus \ref{2conn:x_2x_6x_7notATriangle} holds.

    \begin{sublemma}\label{2conn:x_4x_6x_7notATriangle}
       When $\{p,x_3,x_5\}$ is a triangle, $M_1$ has no triangle containing $\{x_4,x_6\}$. 
    \end{sublemma}
    Assume that $\{x_4,x_6,x_7\}$ is a triangle of $M$. Then $D^* = \{x_2,x_3,x_5,x_8\}$ for some element $x_8$. By orthogonality, $x_8$ is neither $x_6$ nor $x_7$. Thus $x_8 \not\in \{p,x_1,x_2,\dots,x_7\}$. Consider $M/x_6\backslash x_7$. By \ref{2conn:x_2x_6x_7notATriangle} and orthogonality, we see that $\{x_4,x_6,x_7\}$ is the only triangle of $M$ containing $x_6$, so $M/x_6\backslash x_7$ is simple. Therefore, it has a triad $T^*$, and $M$ has $T^* \cup x_7$ as a cocircuit that avoids $x_6$. Then $x_4 \in T^*$ so, by orthogonality, $x_1 \in T^*$. As $\{x_1,x_2,x_3\}$ and $\{x_2,x_4,x_5\}$ are triangles, it follows by orthogonality that $\{x_1,x_2,x_4,x_7\}$ is a cocircuit. However, as $\{x_1,x_2,x_4,x_6\}$ is a cocircuit, $M^*|\{x_1,x_2,x_4,x_6,x_7\} \cong U_{3,5}$, a contradiction. Thus \ref{2conn:x_4x_6x_7notATriangle} holds.

    By \ref{2conn:x_2x_6x_7notATriangle} and \ref{2conn:x_4x_6x_7notATriangle}, $\{p,x_3,x_5\}$ is not a triangle of $M_1$. Thus the only triangle containing $p$ in $M_1$ is $\{p,x_1,x_4\}$. As $x_4$ is in a triangle of $M$, it follows that 
    \begin{sublemma}\label{4.2.2half}
         $M$ has $\{x_2,x_4,x_5\}$ as a triangle for some $x_5$ in $\cl_{M_1}(\{x_1,x_2,x_4\})$, or $M$ has $\{x_4,x_5,x_6\}$ as a triangle for some $x_5$ and $x_6$ not in $\cl_{M_1}(\{x_1,x_2,x_4\})$. 
    \end{sublemma}

   Next we eliminate the first possibility.

    \begin{sublemma}\label{2conn:notx_2x_4x_5}
        $\{x_2,x_4,x_5\}$ is not a triangle of $M$.
    \end{sublemma}

        Assume that $\{x_2,x_4,x_5\}$ is a triangle. Let $C^*$ be a $4$-cocircuit of $M$ containing $x_4$. Then, by orthogonality, $C^*$ contains $x_2$ or $x_5$. Moreover, as $\{p,x_1,x_4\}$ is a circuit of $M_1$, it follows that $M$ has a circuit that meets $E(M_1)$ in $\{x_1,x_4\}$. We deduce, by orthogonality, that $x_1\in C^*$. Thus $C^*$ contains $\{x_1,x_2,x_4\}$ or $\{x_1,x_4,x_5\}$. Suppose $C^* \subseteq \{x_1,x_2,x_3,x_4,x_5\}$. Then $$\lambda_{M_1}(\{p,x_1,x_2,\dots,x_5\}) \leq 3+ (6-2) -6 = 1.$$ Since $M$ has no triads, $|E(M_1)| \geq 7$. Thus, as $M_1$ is $3$-connected, it has exactly one element, say $e$, not in $\{p,x_1,x_2,\dots,x_5\}$. Since $M_1$ has a $\{p,x_1,x_3\}$ as a triad, $e \in \cl_{M_1}(\{x_2,x_4\})$. Thus $M_1$ and hence $M$ has a $U_{2,5}$-minor, a contradiction. We deduce that $C^*$ contains an element $x_6$ that is not in $\{p,x_1,x_2,\dots,x_5\}$. Then $C^*$ does not contain $\{x_1,x_4,x_5\}$ by orthogonality with $\{x_1,x_2,x_3\}$, so $C^* = \{x_1,x_2,x_4,x_6\}$. Now $C^* \cap \{p,x_1,x_2,\dots,x_5\}$ is a union of cocircuits and hence is a cocircuit of $M_1|\{p,x_1,x_2,\dots,x_5\}$. But $\{x_1,x_2,x_4\}$ is not a cocircuit of the last matroid otherwise $M_1$ has $\{p,x_3,x_5\}$ as a triangle, a contradiction. Thus \ref{2conn:notx_2x_4x_5} holds.

    Following \ref{4.2.2half}, the rest of the proof of this lemma will be devoted to proving the following.
    
    \begin{sublemma}\label{2conn:notx_4x_5x_6}
       $M$ has no triangle of the form $\{x_4,x_5,x_6\}$ where $x_5$ and $x_6$ are not in $\cl_{M_1}(\{x_1,x_2,x_4\})$.
    \end{sublemma}
    
    Assume that $M$ does have such a triangle. Then, as $x_4$ is in a $4$-cocircuit of $M$, we may assume that either $\{x_1,x_2,x_4,x_5\}$ or $\{x_1,x_3,x_4,x_5\}$ is a cocircuit of $M$. Thus $$\lambda_{M_1\backslash x_5}(\{p,x_1,x_2,x_3,x_4\}) \leq 3 + (5-2) -5 \leq 1.$$ Suppose $\co(M_1\backslash x_5)$ is $3$-connected. Then  $\{p,x_1,x_2,x_3,x_4\}$ or $E(M_1\backslash x_5) - \{p,x_1,x_2,x_3,x_4\}$ is a series class of $M_1 \backslash x_5$. In each case, since we must have that $|E(M_1)| \geq 9$, we get that $M_1$ has a triad avoiding $p$, a contradiction. Thus, by Bixby's Lemma, $\si(M/x_5)$ is $3$-connected. Observe that $\si(M_1/x_5)$ is $M_1/x_5\backslash x_4$ or is $M_1 /x_5\backslash x_2,x_4$ where the latter occurs when $M_1$ has $\{x_2,x_5,x_7\}$ as a triangle and $\{x_1,x_2,x_4,x_5\}$ as a cocircuit where $x_7 \not \in \{p,x_1,x_2,\dots,x_6\}$.
    
    Continuing the proof of \ref{2conn:notx_4x_5x_6}, we show next that 
    \begin{sublemma}\label{subsub}
        $M_1/x_5 \backslash x_4$ is not simple.
    \end{sublemma}
    
    Assume that $M_1 /x_5 \backslash x_4$ is simple. Then $M/x_5\backslash x_4$ has a triad $T^*$, otherwise the choice of $M$ is contradicted. Then $M$ has $T^* \cup x_4$ as a $4$-cocircuit that avoids $x_5$. Thus, by orthogonality, $x_1 \in T^*$, $x_6 \in T^*$, and $x_2$ or $x_3 \in T^*$. Hence $\{x_1,x_2,x_4,x_6\}$ or $\{x_1,x_3,x_4,x_6\}$ is a cocircuit of $M$. We cannot have $\{x_1,x_2,x_4,x_6\}$ and $\{x_1,x_2,x_4,x_5\}$ as cocircuits of $M$, or both $\{x_1,x_3,x_4,x_5\}$ and $\{x_1,x_3,x_4,x_6\}$ as cocircuits of $M$, otherwise $M^*$ has $U_{3,5}$ as a restriction, a contradiction. Thus either both $\{x_1,x_2,x_4,x_5\}$ and $\{x_1,x_3,x_4,x_6\}$ are cocircuits of $M$, or both  $\{x_1,x_2,x_4,x_6\}$   and $\{x_1,x_3,x_4,x_5\}$ are cocircuits of $M$. Using either of these pairs of $4$-cocircuits, we eliminate $x_4$ to get a cocircuit contained in $\{x_1,x_2,x_3,x_5,x_6\}$. By orthogonality with a circuit of $M$ containing $\{x_1,x_4\}$ and elements of $E(M_2) - p$, we deduce that $\{x_2,x_3,x_5,x_6\}$ is a cocircuit of $M$. As $M_1$ also has $\{p,x_1,x_3\}$ and either $\{x_1,x_2,x_4,x_5\}$ or $\{x_1,x_3,x_4,x_5\}$ as a cocircuit, $$\lambda_{M_1}(\{p,x_1,x_2,\dots,x_6\}) \leq 4 + (7-3) - 7 = 1.$$ This is a contradiction since it implies that $M_1$ has a triad containing $\{x_5,x_6\}$ that avoids $p$.  We conclude that \ref{subsub} holds.

    We now know that $\si(M_1/x_5)$ is $M_1/x_5\backslash x_2,x_4$, and $M_1$ has $\{x_2,x_5,x_7\}$ as a triangle and $\{x_1,x_2,x_4,x_5\}$ as a cocircuit. Now $M_1|\{p,x_1,x_2,\dots,x_7\}$ has rank four and has $\{x_1,x_2,x_4,x_5\}$ as a cocircuit. Thus $r(\{p,x_3,x_6,x_7\}) = 3$, so

    \begin{sublemma}\label{new4.2.6}
       $\{p,x_3,x_6,x_7\}$ is a cocircuit of $M_1$.
    \end{sublemma}

    We show next that

    \begin{sublemma}\label{suburb}
         $M_1$ has no triangle containing $\{x_6,x_7\}$.
    \end{sublemma}

    Assume that $M_1$ has a triangle $\{x_6,x_7,e\}$. Since $M_1$ has no triangle containing $\{p,x_3\}$, we deduce that $e \not\in \cl_{M_1}(\{p,x_1,x_3\})$. Then the simplification of $(M|\{p,x_1,x_2,\dots,x_7,e\}) / e $ is $O_7$ and has $\{p,x_1,x_3\}$ as a triad. Replacing $M_1$ by this copy of $O_7$ contradicts the minimality of $M$. Thus \ref{suburb} holds.

    We now show that 

    \begin{sublemma}\label{subtwo}
        $r(M_1) \geq 5$.
    \end{sublemma}

    Suppose that $r(M_1) \leq 4$. Then $r(M_1) = 4$. The cocircuits $\{p,x_1,x_3\}$ and $\{x_1,x_2,x_4,x_5\}$ of $M_1$ imply that every element of $M_1$ not in these cocircuits is in $\cl_{M_1}(\{x_6,x_7\})$. As $\{x_2,x_3,x_7\}$ is not a triad of $M_1$, there is at least one element in $\cl_{M_1}(\{x_6,x_7\}) - \{x_6,x_7\}$, a contradiction to \ref{suburb}. Thus \ref{subtwo} holds.

    Now $M$ has a $4$-cocircuit $C^*_6$ containing $x_6$. We next show that 

    \begin{sublemma}\label{wheres6}
        $C^*_6 = \{x_5,x_6,x_7,x_8\}$ for some element $x_8 \not\in \{p,x_1,x_2,\dots,x_7\}$.
    \end{sublemma}

        By orthogonality, $C^*_6$ contains $\{x_2,x_5,x_6\}, \{x_5,x_6,x_7\}, \{x_1,x_2,x_4,x_6\}$,  or $\{x_1,x_3,x_4,x_6\}$. The third possibility is eliminated because $\{x_1,x_2,x_4,x_5\}$ is a cocircuit. If the fourth possibility holds, then $$\lambda_{M_1}(\{p,x_1,x_2,\dots,x_6\}) \leq 4 + (7-3) - 7 = 1,$$ so $r(M_1) = 4$, a contradiction to \ref{subtwo}. Suppose $C^*_6 \supseteq \{x_2,x_5,x_6\}$. Then, by orthogonality, $C_6^* = \{x_2,x_3,x_5,x_6\}$. Using the last lambda calculation, we again  obtain the contradiction that $r(M_1) = 4$. We conclude that $C_6 ^* = \{x_5,x_6,x_7,x_8\}$ for some element $x_8$. Moreover, by orthogonality, $x_8 \not\in \{p,x_1,x_2,\dots,x_7\}$. Thus \ref{wheres6} holds.

    Recall that $M/x_5 \backslash x_2,x_4$ is $3$-connected. By the minimality of $M$, it follows that $M / x_5 \backslash x_2,x_4$ has a triad $T^*$. Then $T^* \cup x_4, T^* \cup x_2$, or $T^* \cup~\{x_2,x_4\}$ is a cocircuit of $M$ where $\{x_2,x_4,x_5\} \cap T^* = \emptyset$. Suppose $T^* \cup x_4$ is a cocircuit of $M$. Then, by orthogonality, $x_1 \in T^*$ and $x_3 \in T^*$.  Moreover, $x_6 \in T^*$, so $\{x_1,x_3,x_4,x_6\}$ is a cocircuit of $M$. Then $$\lambda_{M_1}(\{p,x_1,x_2,\dots,x_6\}) \leq 4 + (7-3) - 7 = 1.$$ Since $\{x_5,x_6,x_7\}$ is not a triad of $M_1$, we deduce that $M_1$ has an element in $\cl_{M_1}(\{x_6,x_7\}) - \{x_6,x_7\}$, a contradiction to \ref{suburb}. Thus $T^* \cup x_4$ is not a cocircuit of $M$. 

    Next assume that $T^* \cup x_2$ is a cocircuit of $M$. Then $x_7 \in T^*$. As $x_4 \not\in T^*$, we see that $x_1 \not\in T^*$, so $x_3 \in T^*$. Hence $\{x_2,x_3,x_7\}$ is in a $4$-cocircuit $F^*$ of $M$ that avoids $\{x_1,x_4,x_5,x_6\}$. Let $F^*= \{x_2,x_3,x_7,x_9\}$, for some element $x_9$ that is not in $\{p,x_1,x_2,\dots,x_7\}$. Now, by Lemma~\ref{minCond}, $x_9$ is in a triangle $T$. By orthogonality, $T$ contains $\{x_2,x_9\}$, $\{x_3,x_9\}$, or $\{x_7,x_9\}$. In the first case, by orthogonality, $T$ is $\{x_1,x_2,x_9\}$, $\{x_2,x_4,x_9\}$, or $\{x_2,x_5,x_9\}$. By orthogonality with the triad $\{p,x_1,x_3\}$ in $M_1$, it follows that $T \not = \{x_1,x_2,x_9\}$. If $T$ is $\{x_2,x_4,x_9\}$ or $\{x_2,x_5,x_9\}$, then $M /x_6$ has an $O_7$-minor having $p$ in its triad, contradicting the minimality of $M_1$. If $T$ contains $\{x_3,x_9\}$, then, by orthogonality with the triad $\{p,x_1,x_3\}$ in $M_1$, we get that $T = \{p,x_3,x_9\}$, a contradiction. We conclude that $T$ contains $\{x_7,x_9\}$.

    Next we show that 

    \begin{sublemma}\label{9not8}
         $x_8 \not = x_9$.
    \end{sublemma}
    
    Assume that $x_8 = x_9$. Then $M_1$ has a cocircuit $J^*$ such that $$ J^* \subseteq(\{x_2,x_3,x_7,x_8\} \cup \{x_5,x_6,x_7,x_8\} )- \{x_8\}.$$ If $x_7 \in J^*$, then, by orthogonality, $T$ contains an element of $\{x_2,x_3,x_5,x_6\}$. Thus $x_8 \in \cl_{M_1}(\{p,x_1,x_3,x_6\})$, so $$\lambda_{M_1}(\{p,x_1,x_2,\dots,x_8\}) \leq 4 + (9-4) - 9 = 0.$$ Hence $E(M_1) = \{p,x_1,x_2,\dots,x_8\})$. Thus $r(M_1) = 4$, a contradiction to \ref{subtwo}. We conclude that \ref{9not8} holds.

    By orthogonality and \ref{suburb}, $T$ is $\{x_7,x_8,x_9\}$. Then $$\lambda_{M_1}(\{p,x_1,x_2,\dots,x_9\}) \leq 5 + (10-4) - 10 = 1.$$ Therefore  $|E(M) - \{p,x_1,x_2,\dots,x_9\}| \leq 1$. By \ref{subtwo}, $r(M_1) \geq 5$. As \linebreak $r(\{p,x_1,x_2,\dots,x_7\})= 4$, we deduce that $E(M_1) - \{p,x_1,x_2,\dots,x_7\}$ is a triad, a contradiction. We deduce, when $T^* \cup x_2$ is not a cocircuit of $M$.
    
Therefore  we now know that $T^* \cup \{x_2,x_4\}$ is a cocircuit of $M$ where $T^* \cap \{x_2,x_4,x_5\} = \emptyset$. Then, by orthogonality, $x_1 \in T^*$, $x_6 \in T^*$, and $x_7 \in T^*$. Thus $M$ has $\{x_1,x_2,x_4,x_6,x_7\}$ as a cocircuit. As $\{x_1,x_2,x_4,x_5\}$ is also a cocircuit, by eliminating $x_1$, from the union of these two cocircuits, we get a cocircuit contained is $\{x_2,x_4,x_5,x_6,x_7\}$. The triangles $\{x_1,x_2,x_3\}$ and $\{p,x_1,x_4\}$ of $M_1$ mean that $\{x_5,x_6,x_7\}$ is a cocircuit of $M_1$, a contradiction. Thus \ref{2conn:notx_4x_5x_6} holds and, by \ref{4.2.2half}, the lemma~follows.
\end{proof}

We may now focus on a $3$-connected minor-minimal simple ternary  matroid $M$ in $\mathcal{M}_4$.

\begin{lemma}\label{No4PtLine}
     Let $M$ be a $3$-connected minor-minimal ternary matroid in $\mathcal{M}_4$. Then $M$ does not have $U_{2,4}$ as a restriction.
\end{lemma}

\begin{proof}
    Assume that  $M|\{x_1,x_2,x_3,x_4\} \cong U_{2,4}$. Since every cocircuit of $M$ has at least four elements, $|E(M)| \geq 8$. By Lemma~$\ref{minCond}$, we may assume that $M$ has $4$-cocircuits $C^*_1$ and $C^*_2$ such that  $\{x_1,x_2,x_3\} \subseteq C_1^*$ and $\{x_2,x_3,x_4\} \subseteq~C_2^*$. Observe that $C_1^* \not= C_2^*$ otherwise $$\lambda(\{x_1,x_2,x_3,x_4\}) \leq 2 + (4-1) - 4 = 1,$$ so $|E(M)| \leq 5$, a contradiction. Now suppose $M$ has a point $x_5$ such that $C_1^* = \{x_1,x_2,x_3,x_5\}$ and $C_2^* = \{x_2,x_3,x_4,x_5\}$. Then $$\lambda(\{x_1,x_2,x_3,x_4,x_5\}) \leq 3 + (5-2) - 5 =  1.$$ However, this implies that $|E(M)| \leq 6$, a contradiction. Therefore $C_1^* = \{x_1,x_2,x_3,x_5\}$ and $C_2^* = \{x_2,x_3,x_4,x_6\}$ for distinct elements $x_5$ and $x_6$ not in $\{x_1,x_2,x_3,x_4\}$. Now, by Lemma~\ref{minCond}, $x_5$ is in a triangle $T$. If $x_2 \in T$, then $T = \{x_2,x_5,x_6\}$. Thus $\lambda(\{x_1,x_2,x_3,x_4,x_5,x_6\}) \leq 3 + (6-2) -6 =  1$, so $|E(M)| \leq 7$, a contradiction. Therefore neither $\{x_2,x_5\}$ nor $\{x_3,x_5\}$ is in a triangle and we may assume that $\{x_1,x_5,x_7\}$ is a triangle for some element $x_7$ that is not in $\{x_1,x_2,\dots,x_6\}.$ Similarly, we may assume that $\{x_4,x_6,x_8\}$ is a triangle for some element $x_8$ that is not in $\{x_1,x_2,\dots, x_6\}$. Suppose $x_7 = x_8$. Then $\lambda(\{x_1,x_2,\dots,x_7\}) \leq 1$. As $M$ has no triads, it follows that $E(M) - \{x_1,x_2,\dots,x_7\} = \{e\}$ for some element $e$, and $r(M) =~3$. As $E(M) -(C_1^* \cup C_2^*)$ is a rank-one set containing $\{x_7,e\}$, we have a contradiction. Thus $x_7$ and $x_8$ are distinct.

    Now, $\lambda_{M\backslash x_5}(\{x_1,x_2,x_3,x_4,x_6\}) \leq 3 + (5-2) - 5 = 1$. As $M \backslash x_5$ has no $2$-cocircuits, $\co(M\backslash x_5) = M \backslash x_5$ and this matroid is not $3$-connected. Thus, by Bixby's Lemma~[\ref{bixby}], $\si(M/x_5)$ is $3$-connected. Observe that $\si(M/x_5) =M/x_5\backslash x_1$. By the minimality of $M$, there is a triad $T^*$ in $ M/x_5\backslash x_1$ and $T^* \cup x_1$ is a cocircuit of $M$ that avoids $x_5$. Then $x_7 \in T^*$ and two of $x_2,x_3$, and $x_4$ are in $T^*$. Observe that $x_4 \not\in T^*$, otherwise $T^*$ also contains $x_6$ or $x_8$, so $|T^*| \geq 4$, a contradiction. Thus $T^* = \{x_2,x_3,x_7\}$, so $T^* \cup x_1 = \{x_1,x_2,x_3,x_7\}$ is a cocircuit of $M$. However, this implies that $M^*|\{x_1,x_2,x_3,x_5,x_7\} \cong U_{3,5}$. As $M$ is ternary, this is a contradiction.
\end{proof}

\begin{lemma}\label{4CoCctNoTriangle}
     Let $M$ be a minor-minimal ternary matroid in $\mathcal{M}_4$. If $M$ has a $4$-cocircuit $D^*$ that contains a triangle, then $M$ is isomorphic to $P_7$ or $H_{12}$. 
\end{lemma}

\begin{proof}
    Assume that $M$ is not isomorphic to $P_7$ or $H_{12}$. Let $\{x_1,x_2,x_3,x_4\}$ be $D^*$ and let $\{x_1,x_2,x_3\}$ be a triangle. Since $M \not\cong H_{12}$, Lemma~\ref{MainTheorem:2conn} implies that $M$ is $3$-connected. Then $\co(M\backslash x_4)$ is not $3$-connected because $M\backslash x_4$ has $\{x_1,x_2,x_3\}$ as a circuit and a cocircuit. Thus $\si(M/x_4)$ is $3$-connected. Now, by Lemma~\ref{minCond}, $x_4$ is in a triangle of $M$. We may assume that $\{x_3,x_4,x_5\}$ is a triangle, for some element $x_5$ that is not in $\{x_1,x_2,x_3,x_4\}$. 

    \begin{sublemma}\label{x4notriangle}
        $x_4$ is in a triangle other than $\{x_3,x_4,x_5\}$.
    \end{sublemma}

    Assume that $x_4$ is in a unique triangle. Consider $M/x_4\backslash x_3$. This matroid is simple, so, by the minimality of $M$, it has a triad $T_3^*$. Thus $T_3^* \cup x_3$ is a cocircuit of $M$ that does not contain $x_4$. By orthogonality, $x_5 \in T^*$. Observe that if $T_3^* \cup x_3 = \{x_1,x_2,x_3,x_5\}$, then $M^*|\{x_1,x_2,x_3,x_4,x_5\} \cong U_{3,5}$, a contradiction. Thus, by symmetry between $x_1$ and $x_2$,  we may assume that $T^*\cup x_3 = \{x_2,x_3,x_5,x_6\}$ for some element $x_6$ not in $\{x_1,x_2,\dots,x_5\}$. Then, by Lemma~\ref{minCond}, $x_6$ is in a triangle $C_6$. Suppose $x_2$ or $x_3$ is in $C_6$. As $x_4 \not\in C_6$, by orthogonality with $\{x_1,x_2,x_3,x_4\}$, we deduce that $C_6$ contains two elements of $\{x_1,x_2,x_3\}$. Thus $M$ has a $4$-point line, a contradiction to Lemma~\ref{No4PtLine}. We now know that $C_6 = \{x_5,x_6,x_7\}$ for some element $x_7$ not in $\{x_1,x_2,\dots,x_7\}$.
    
    Consider $M / x_2 \backslash x_1$. Because $\{x_3,x_4,x_5\}$ is the unique triangle of $M$ containing $x_4$ and $M$ has no $4$-point lines, $M / x_2 \backslash x_1$ is simple. By the minimality of $M$, it follows that $M / x_2 \backslash x_1$ has a triad $T_1^*$. Thus $T^*_1 \cup x_1$ is a cocircuit of $M$ avoiding $x_2$. By orthogonality, $x_3 \in T^*_1$, and $x_4$ or $x_5$ is in $T_1^*$. Suppose $T_1^* \cup x_1  = \{x_1,x_3,x_4,e\}$ for some element $e$. Then $e \not = x_2$ and $M^*|\{x_1,x_2,x_3,x_4,e\} \cong U_{3,5}$, a contradiction. Thus $T_1^* \cup x_1$ contains $\{x_1,x_3,x_5\}$ so, by orthogonality, $T_1^* \cup x_1$ is $\{x_1,x_3,x_5,x_6\}$ or $\{x_1,x_3,x_5,x_7\}$. In the first case, we get the contradiction that $M^*|\{x_1,x_2,x_3,x_5,x_6\} \cong U_{3,5}$. Hence $T^*_1 \cup x_1 = \{x_1,x_3,x_5,x_7\}$. Thus $$\lambda(\{x_1,x_2,\dots,x_7\}) \leq r(\{x_1,x_2,\dots,x_7\}) + (7-3) - 7 = r(\{x_1,x_2,\dots,x_7\}) - 3.$$ If $r(\{x_1,x_2,\dots,x_7\}) = 3$, then $E(M) = \{x_1,x_2,\dots,x_7\}$ and,  from the known circuits and cocircuits, we obtain the contradiction that $M \cong P_7$. Thus $r(\{x_1,x_2,\dots,x_7\}) = 4$, so $M$ has at most one element not in \linebreak $\{x_1,x_2,\dots,x_7\}$. As $r(\{x_1,x_2,x_3,x_4,x_5\}) = 3$, we deduce that $M$ has a cocircuit of size less than four, a contradiction. We conclude that \ref{x4notriangle} holds.

    Now, by orthogonality, we may assume that $M$ has $\{x_2,x_4,x_6\}$ as a triangle for some element $x_6$ not in $\{x_1,x_2,\dots, x_5\}$. Assume that $x_1$ is in a triangle $T_1$ other than $\{x_1,x_2,x_3\}$. Then, by Lemma~\ref{No4PtLine} and orthogonality, $T_1 = \{x_1,x_4,x_7\}$ for some element $x_7$ not in $\{x_1,x_2,\dots, x_6\}$. Now $M$ has a $4$-cocircuit $D^*_5$ containing $x_5$. By orthogonality, $D^*_5 \subseteq \{x_1,x_2,\dots, x_7\}$. Thus $\lambda(\{x_1,x_2,\dots,x_7\}) \leq 3 + (7-2) - 7 = 1$. Hence $M$ has at most one element not in $\{x_1,x_2,\dots,x_7\}$, and $r(M) = 3$. As $\{x_1,x_2,x_3,x_4\}$ is a cocircuit of $M$, it follows that $E(M) - \{x_1,x_2,x_3,x_4\}$ is a line of $M$. By Lemma~\ref{No4PtLine}, this line has exactly three points, that is, $\{x_5,x_6,x_7\}$ is a triangle. Hence $M\cong P_7$, a contradiction. We deduce that $\{x_1,x_2,x_3\}$ is the only triangle containing $x_1$. Then $M / x_1 \backslash x_2$ is simple. Hence this matroid has a triad $T^*_2$ avoiding $x_1$, so $T^*_2 \cup x_2$ is a cocircuit of $M$. By orthogonality, $x_3 \in T^*_2$. If $x_4 \in T^*_2$, then $M^*|(\{x_1,x_2,x_3,x_4\}  \cup T^*_2) \cong U_{3,5}$, a contradiction. Thus $x_4 \not \in T^*_2$. Then, by orthogonality, $T^*_2 \cup x_2 = \{x_2,x_3,x_5,x_6\}$. Thus $\lambda(\{x_1,x_2,\dots,x_6\}) \leq 3 + (6-2) - 6 = 1$. Hence $r(M) = 3$ and $|E(M)| = 7$. Let $x_7$ be the element of $E(M) - \{x_1,x_2,\dots,x_6\}$. Then $\{x_1,x_4,x_7\}$ is the complement in $M$ of the cocircuit $\{x_2,x_3,x_5,x_6\}$. Thus $\{x_1,x_4,x_7\}$ is a triangle, a contradiction.
\end{proof}

\begin{lemma}\label{4CocctInd}
     Let $M$ be a $3$-connected minor-minimal ternary matroid in $\mathcal{M}_4$. If $M$ has a $4$-cocircuit that is also a $4$-circuit, then $M$ is isomorphic to $F_7^-$ or $P_7$. 
\end{lemma}

\begin{proof}
 Let $X = \{x_1,x_2,x_3,x_4\}$ and assume that $X$ is a circuit and a cocircuit of $M$. Assume that $r(M) \geq 4$. Take $y$ in $E(M) - \cl(X)$. Then $y$ is in a triangle $Y$ of $M$. Since $y \not\in \cl(X)$, by orthogonality, $X \cap Y = \emptyset$. Observe that $\lambda(X) = 3 + 3 -4 = 2 = \kappa_M(X,Y)$. By Theorem \ref{GGWTutte}, $M$ has a minor $N$ such that $\kappa_N(X,Y) = 2$ and $M|X = N|X$ while $M|Y = N|Y$. Then \linebreak $r_N(Y) = r_M(Y) = 2$ and $r_N(X) = r_M(X) = 3$. As $2 = r_N(X) + r_N(Y) -r(N)$, we deduce that $r(N) = 3$. Thus $N$ is a rank-$3$ simple ternary matroid having $X$ as a circuit and $Y$ as a hyperplane. As $N$ has no $4$-point lines, it follows that $N$ is isomorphic to $F_7^-$ or $P_7$. By the minimality of $M$, we obtain a contradiction. 

 We now know that $r(M) = 3$. Then $E(M) - X$ is a hyperplane  $Y$ of $M$. By Lemma~\ref{No4PtLine},  $Y$ cannot have more than three elements. Suppose $Y$ has exactly two elements. Then $M$ has no triangles, otherwise it has a triad. Thus $M \cong U_{3,6}$, a contradiction. We deduce that $M$ is a rank-$3$ ternary matroid whose ground set is the disjoint union of a triangle $Y$ and a set that is both a circuit and a cocircuit. Thus $M$ is isomorphic to $F_7^-$ or $P_7$.
\end{proof}

For the rest of the section, every time we consider a $4$-cocircuit $C^*$ of $M$, we may assume that $r(C^*) = 4$. The following lemma shows that we cannot have two $4$-cocircuits contained in the disjoint union of two triangles.

\begin{lemma}\label{C1C2_intersection_not_2}
     Let $M$ be a $3$-connected minor-minimal ternary matroid in $\mathcal{M}_4$ and suppose that $M$ is not isomorphic to $P_7, M^*(K_{3,3}),$ or $Q_9$. Let $C^* = \{x_1,x_2,x_4,x_5\}$ be a cocircuit of $M$. Let $\{x_1,x_2,x_3\}$ and $\{x_4,x_5,x_6\}$ be disjoint triangles of $M$. Then there is no  $4$-cocircuit other than $C^*$ that meets both $\{x_1,x_2,x_3\}$ and $\{x_4,x_5,x_6\}$.
\end{lemma}

\begin{proof}
    Assume there is such a $4$-cocircuit $D^*$. By Lemma~\ref{4CocctInd}, $r(D^*)~=~4$.  Assume that $\{x_1,x_2\} \subseteq D^*$. As $D^*$ contains two members of $\{x_4,x_5,x_6\}$, it follows that $M^*|(C^* \cup D^*) \cong U_{3,5}$, a contradiction. Therefore, $\{x_1,x_3\} \subseteq D^*$ or $\{x_2,x_3\} \subseteq D^*$. Furthermore, by symmetry, one of $\{x_4,x_6\}$ or $\{x_5,x_6\}$ is contained in $D^*$. Thus we may assume that $D^* = \{x_2,x_3,x_5,x_6\}$. Let $X = \{x_1,x_2,\dots,x_6\}$. By Lemma~\ref{4CocctInd}, $r(X) \not = 3$, so $r(X) = 4$. Thus $\lambda(X) \leq 4 + (6-2) - 6 = 2$.  We next show that 

    \begin{sublemma} \label{rm4}
        $r(M) = 4$.
    \end{sublemma}
    Assume that $r(M) \geq 5$ and take $y \in E(M) - \cl(X)$. Then $y$ is in a triangle $Y$ of $M$, and $Y$ avoids $X$. We have that $\kappa_M(X,Y) = \lambda(X) = 2$. Thus, by Theorem \ref{GGWTutte}, $M$ has a minor $N$ with ground set $X \cup Y$ such that $\kappa_N(X,Y) = 2$ while $N|X = M|X$ and $N|Y = M|Y$. Thus $r_N(X) = 4$ and $r_N(Y) = 2$, so $r(N) = 4$. Let $Y = \{x_7,x_8,x_9\}$. Then $\{x_1,x_2,x_3\}$ and $\{x_4,x_5,x_6\}$ are triangles of $N$. Moreover, since $\{x_1,x_2,x_4,x_5\}$ and $\{x_2,x_3,x_5,x_6\}$ are cocircuits of $M$, each of these sets is a union of cocircuits of $N$. Because $r_N(\{x_7,x_8,x_9\}) = 2$, it follows that $\{x_1,x_2,x_4,x_5\}$ and $\{x_2,x_3,x_5,x_6\}$ are cocircuits of $N$. Because $N$ is a simple minor of $M$, it follows that $N$ has a cocircuit of size less than four. Suppose $N$ has a $2$-cocircuit $Z$. Then $Z$ is contained in one of $\{x_1,x_2,x_3\}, \{x_4,x_5,x_6\}$, or $\{x_7,x_8,x_9\}$, and $N \backslash Z$ is a $7$-point plane of $N$. As $r(\{x_1,x_2,\dots,x_6\}) = 4$, we may assume that $Z \subseteq \{x_1,x_2,x_3\}$. Then $N \backslash (Z \cup \{x_1,x_2,x_4,x_5\})$ is a line, $\{x_6,x_7,x_8,x_9\}$, of $N$, a contradiction to orthogonality with the cocircuit $\{x_2,x_3,x_5,x_6\}$. Hence $N$ is cosimple. Thus $N$ has a triad. By orthogonality and the fact that $r_N(X) = 4$, we may assume that $\{x_1,x_2,x_3\}$ is a triad of $N$. Then deleting $\{x_1,x_2,\dots,x_6\}$ from $N$ produces a rank-one matroid with ground set $\{x_7,x_8,x_9\}$, a contradiction. We conclude that \ref{rm4} holds.
    
    As $r(M) = 4$ and $M$ has a plane with at least four points, it follows that $|E(M)| \geq 8$. Let $x_7$ be an element of $E(M) - \{x_1,x_2,\dots,x_6\}$. Then $x_7$ is in a triangle $T$ of $M$. If $T$ avoids $\{x_1,x_2,\dots,x_6\}$, then $|E(M)| \geq 9$. If $T$ meets $\{x_1,x_2,\dots,x_6\}$, then $M$ has a plane with at least five points, so $|E(M)| \geq 9$. Let $E(M) - \{x_1,x_2,\dots,x_6\} = \{x_7,x_8,\dots, x_n\}$ for some $n \geq 9$. As $E(M) - (\{x_1,x_2,x_4,x_5\} \cup \{x_2,x_3,x_5,x_6\})$ is a flat of $M$ of rank at most two and this flat contains $\{x_7,x_8,\dots,x_n\}$, we deduce, by Lemma~\ref{No4PtLine}, that $n = 9$ and $\{x_7,x_8,x_9\}$ is a rank-$2$ flat of $M$. Since $M$ has no triads, $r(\{x_7,x_8,x_9\} \cup L) = 4$ for each $L$ in $\{\{x_1,x_2,x_3\},\{x_4,x_5,x_6\}\}$.

    Since $\{x_1,x_2,x_4,x_5\}$ and $\{x_2,x_3,x_5,x_6\}$ are cocircuits of $M$, the sets $\{x_1,x_4,x_7,x_8,x_9\}$ and $\{x_3,x_6,x_7,x_8,x_9\}$ are planes of $M$. As $r(X) = 4$, we see that $\{x_1,x_3,x_4,x_6\}$ is not a circuit of $M$. Since $M$ is ternary, in the rank-$4$ ternary projective geometry $P$ of which $M$ is a restriction, $\cl_P(\{x_7,x_8,x_9\})$ has a single point $z_3$ that is not in $\{x_7,x_8,x_9\}$. Thus, by symmetry, we may assume that $\{x_3,x_6,x_9\}$ is a triangle of $M$. Moreover, by symmetry again, $\{x_1,x_4,z_3\}$ or $\{x_1,x_4,x_7\}$ is a triangle of $P$.

    We may assume that $\{x_1,x_4,z_3\}$ is a triangle of $P$, otherwise, by \linebreak Lemma~\ref{newboy}, we obtain the contradiction that $M \cong M^*(K_{3,3})$ or $M$ has $P_7$ as a minor. Lemma~\ref{newboy} also implies that $\{x_7,x_8,x_9\}$ is the unique triangle of $M$ containing $e$ for each $e$ in $\{x_7,x_8\}$. Now $M /x_7\backslash x_8$ has rank three and has $\{x_1,x_2,x_3\}, \{x_4,x_5,x_6\}$, $\{x_3,x_6,x_9\}$ and $\{x_1,x_4,x_9\}$ as triangles. Since $P_7 \not\cong M/x_7 \backslash x_8$, the last matroid, which is ternary, has $\{x_1,x_2,x_3,x_5\}$ or $\{x_2,x_4,x_5,x_6\}$ as a line. Thus $M$ has $\{x_1,x_2,x_3,x_5,x_7\}$ or $\{x_2,x_4,x_5,x_6,x_7\}$ as a rank-$3$ set. Then, by considering $M /x_8 \backslash x_7$ instead of $M /x_7 \backslash x_8$, we deduce that $M$ has $\{x_1,x_2,x_3,x_5,x_8\}$ or $\{x_2,x_4,x_5,x_6,x_8\}$ as a plane. Now $M$ cannot have both $\{x_1,x_2,x_3,x_5,x_7\}$ and $\{x_1,x_2,x_3,x_5,x_8\}$ as planes or $M$ has a triad. Therefore either both $\{x_1,x_2,x_3,x_5,x_7\}$ and $\{x_2,x_4,x_5,x_6,x_8\}$ are planes of $M$ or both $\{x_1,x_2,x_3,x_5,x_8\}$ and \linebreak $\{x_2,x_4,x_5,x_6,x_7\}$ are planes of $M$. The first of these possibilities gives that $P$ has $\{x_5,x_7,z_1\}$ and $\{x_2,x_8,z_2\}$ as triangles where $\{x_1,x_2,x_3,z_1\}$ and $\{x_4,x_5,x_6,z_2\}$ are lines of $P$ for some points  $z_1$ and $z_2$ of $P$. The second case is symmetric. In both cases, $M\cong Q_9$,  a contradiction.
\end{proof}

\begin{lemma}\label{newton}
    Let $M$ be a  minor-minimal ternary matroid in $\mathcal{M}_4$. Let $x_1,x_2,\dots,x_6, $ and $x_7$ be distinct elements of $E(M)$ such that $\{x_1,x_2,x_3\}$, $\{x_1,x_4,x_6\}$, and $\{x_1,x_5,x_7\}$ are triangles of $M$, and $\{x_1,x_2,x_4,x_5\}$ and $\{x_1,x_3,x_6,x_7\}$ are cocircuits of $M$. Then $M$ is isomorphic to $F_7^-,P_7,$ or $M(K_5)$.
\end{lemma}

\begin{proof}
    Let $X = \{x_1,x_2,\dots,x_7\}$. We may assume that $r(X) \geq 4$, otherwise $M \cong P_7$. Assume $r(X) =  4$. Then $|E(M) - X| \geq 2$ as $M \in \mathcal{M}_4$ and $X$ contains a plane with at least five points. Suppose $|E(M) - X| = 2$ and $y \in E(M) - X$. Then $y$ is in a triangle that, by symmetry, has $x_3$ and $x_6$ as its other elements. Then $M$ has a plane with at least six elements, so $|E(M) - X| \geq 4$, a contradiction. We may now assume that $|E(M) - X| \geq 3$. As $\lambda(X) = 2$, we see that $r(E(M) - X) = 2$. Then, by Lemma \ref{No4PtLine}, $E(M) - X$ is a triangle of $M$. When $r(M) = 4$, we call this triangle $Y$. When $r(M) \geq 5$, take $z$ to be an element of $E(M) - \cl(X)$. In this case, we let $Y$ be this triangle. Thus, whenever $r(M) \geq 4$, we have a triangle of $M$ that avoids $X$. Now $\kappa_M(X,Y) = \lambda(Y) = 2$, so, by Theorem \ref{GGWTutte}, $M$ has a minor $N$ with $E(N) = X \cup Y$ and $\kappa_N(X,Y) = 2$ such that $M|X = N|X$ and $M|Y = N|Y$. Thus $r_N(Y) = 2$ and $r_N(X) = r(N) = 4$. Now each of $\{x_1,x_2,x_4,x_5\}$ and $\{x_1,x_3,x_6,x_7\}$ is a cocircuit of $N$, otherwise $r_N(Y) \leq 1$, a contradiction. 

    Let $P$ be the rank-$4$ ternary projective space of which $N$ is a restriction. Let $\cl_P(Y) = \{a,b,c,d\}$. Now $N$ has $Y \cup \{x_2,x_4,x_5\}$ and $Y\cup \{x_3,x_6,x_7\}$ as planes that meet on the line $Y$. Let $\{x_2,x_4,a\}, \{x_2,x_5,b\}$, and $\{x_4,x_5,c\}$ be triangles of $P$. Note that $r(\{x_2,x_3,x_4,x_6\}) = 3$. Now the projective line $\cl_P(Y)$ meets the projective plane $\cl_P(\{x_2,x_3,x_4,x_6\})$ in the point $a$ because $\cl_P(Y)$ is not contained in the projective plane. Thus $\{x_3,x_6,a\}$ is a triangle of $P$. Similarly, $\{x_3,x_7,b\}$ and $\{x_6,x_7,c\}$ are triangles of $P$. If all of $a,b,$ and $c$ are in $Y$, then $N\backslash x_1 \cong M(K_5\backslash e)$. The triangles containing $x_1$ imply that $N \cong M(K_5)$. We may now assume that exactly two of $a,b,$ and $c$ are in $Y$. Thus $d \in Y$ and $N/d$ is a rank-$3$ matroid having three $3$-point lines containing $x_1$ and having $\{x_2,x_4,x_5\}$ and $\{x_3,x_6,x_7\}$ as triangles. Thus $N / d$ has a $P_7$-minor, a contradiction.
\end{proof}

 \begin{lemma}\label{NoThreeTriangles}
         Let $M$ be a minor-minimal ternary matroid in $\mathcal{M}_4$. If $M$ is $3$-connected, then no element of $M$ is in at least three triangles unless $M$ is isomorphic to $F_7^-,P_7,$ or $M(K_5)$.
\end{lemma}

\begin{proof}
   Assume that $M$ is not isomorphic to $F_7^-,P_7,$ or $M(K_5)$. Recall that, by Lemma~\ref{4CocctInd}, every $4$-cocircuit is independent. Let $\{x_1,x_2,x_4,x_5\}$ be a cocircuit $C^*$ of $M$. Assume that $x_1$ is in more than three triangles. Let $T_1,T_2,T_3,$ and $T_4$ be four such triangles. Then, by Lemma~\ref{4CoCctNoTriangle}, $|C^* \cap T_i| = 2$ for all $i$ in $\{1,2,3,4\}$. However, this would imply that $x_1$ is contained in a four-point line, a contradiction to Lemma~\ref{No4PtLine}. Thus we may assume that $x_1$ is in exactly three triangles, say, $\{x_1,x_2,x_3\}$, $\{x_1,x_4,x_6\}$, and $\{x_1,x_5,x_7\}$. We first show the following.
    \begin{sublemma}\label{3triangles:x4oneTriangle}
        $x_4$ is in a triangle other than $\{x_1,x_4,x_6\}$.
    \end{sublemma}

  Assume that this fails. Then $M/x_4\backslash x_1$ is simple. By the minimality of $M$, it follows that $M/x_4\backslash x_1$ has a triad $T^*$. Thus $T^* \cup x_1$ is a cocircuit of $M$ where $x_4 \not\in T^*$. Then $x_6 \in T^*$ and, by orthogonality, $T^*$ contains $x_2$ or $x_3$, and $T^*$ contains $x_5$ or $x_7$. If $\{x_1,x_2,x_5,x_6\}$ is a cocircuit, then $M^*|\{x_1,x_2,x_4,x_5,x_6\} \cong U_{3,5}$, a contradiction. If $\{x_1,x_3,x_5,x_6\}$ is a cocircuit, then, as $\{x_1,x_2,x_4,x_5\}$ is a cocircuit, by circuit elimination, we deduce that $M$ has a cocircuit contained in $\{x_2,x_3,x_4,x_5,x_6\}$. By orthogonality with the triangle $\{x_1,x_5,x_7\}$, we see that $\{x_2,x_3,x_4,x_6\}$ is a cocircuit. However, $\{x_2,x_3,x_4,x_6\}$ is also a circuit, a contradiction to Lemma~$\ref{4CocctInd}$. Now assume that $\{x_1,x_2,x_6,x_7\}$ is a cocircuit. Then, again, as $\{x_1,x_2,x_4,x_5\}$ is a cocircuit, there is a cocircuit contained in $\{x_2,x_4,x_5,x_6,x_7\}$. By orthogonality with the triangle $\{x_1,x_2,x_3\}$, we deduce that $\{x_4,x_5,x_6,x_7\}$ is a cocircuit. As $\{x_4,x_5,x_6,x_7\}$  is also a circuit, we get a contradiction to Lemma~\ref{4CocctInd}. We conclude that $\{x_1,x_3,x_6,x_7\}$ is a cocircuit of $M$. By Lemma~\ref{newton}, this yields a contradiction. Thus \ref{3triangles:x4oneTriangle} holds.

  By \ref{3triangles:x4oneTriangle} and symmetry, $\{x_4,x_5,x_8\}$ is a triangle of $M$ for some element $x_8$ not in $E(M) - \{x_1,x_2,\dots,x_7\}$. By symmetry again, $x_2$ is in a triangle other than $\{x_1,x_2,x_3\}$, so $M$ has $\{x_2,x_4,x_9\}$ or $\{x_2,x_5,x_9\}$ as a triangle for some element $x_9$. By applying the permutation $(x_4,x_5)(x_6,x_7)$ to $E(M)$, we deduce that these two cases are symmetric, so it suffices to consider the former. 

  Suppose first that $x_1,x_2,\dots,x_8,$ and $x_9$ are distinct. Now $M$ has a cocircuit $C^*_6$ that contains $x_6$. By orthogonality, $x_1$ or $x_4$ is in $C^*_6$, so $C^*_6$ is $\{x_1,x_3,x_6,x_7\}$ or $\{x_4,x_6,x_8,x_9\}$. The first case gives a contradiction by Lemma~\ref{newton}. In the second case, $M$ has $\{x_1,x_4,x_6\}$, $\{x_2,x_4,x_9\},$ and $\{x_4,x_5,x_8\}$ as triangles and has $\{x_1,x_2,x_4,x_5\}$ and $\{x_4,x_6,x_8,x_9\}$ as cocircuits and again we get a contradiction by Lemma~\ref{newton}. 

  We may now assume that $x_1,x_2,\dots,x_8$, and $x_9$ are not distinct. Since $x_1,x_2,\dots,x_6,$ and $x_7$ are distinct, $x_8$ or $x_9$ is in $\{x_1,x_2,\dots,x_7\}$. But one easily checks that each possibility implies that $r(\{x_1,x_2,\dots,x_7\}) = 3$, a contradiction. We conclude that Lemma~\ref{NoThreeTriangles} holds.  
\end{proof}

\begin{lemma}\label{4co}
    Let $M$ be a minor-minimal  matroid in $\mathcal{M}_4$. If $M$ has \linebreak $\{x_1,x_2,x_3\}$ as the unique triangle containing $x_1$, then $M$ has a $4$-cocircuit that meets $\{x_1,x_2,x_3\}$ in $\{x_2,x_3\}$.
\end{lemma}

\begin{proof}
    Consider $M /x_1 \backslash x_3$. Since it is simple, by the minimality of $M$, the matroid $M / x_1 \backslash x_3$ has a triad $T^*$. Thus $T^* \cup x_3$ is a cocircuit of $M$. By orthogonality, $x_2 \in T^*$ since $x_1 \not\in T_1^*$. Hence the $4$-cocircuit $T^*\cup x_3$ meets $\{x_1,x_2,x_3\}$ in $\{x_2,x_3$\}.
\end{proof}

\begin{lemma}\label{3tri}
    Let $M$ be a minor-minimal ternary matroid in $\mathcal{M}_4$ that is not isomorphic to $M^*(K_{3,3})$ or $Q_9$. Assume that $M$ has $\{x_1,x_2,x_3\},\linebreak\{x_4,x_5,x_6\}$, and $\{x_7,x_8,x_9\}$ as disjoint triangles and has no other triangles meeting $\{x_1,x_2,\dots,x_9\}$. Then $M$ does not have all of $\{x_1,x_2,x_4,x_5\},\linebreak \{x_2,x_3,x_7,x_8\}$, and $\{x_5,x_6,x_8,x_9\}$ as cocircuits.
\end{lemma}

\begin{proof}
    Assume that $M$ does have the three specified sets as cocircuits. By Lemma~\ref{4co}, $M$ has a $4$-cocircuit $C^*$ containing $\{x_1,x_3\}$. Thus, by Lemmas~\ref{4CoCctNoTriangle}~and~\ref{4CocctInd}, $C^* = \{x_1,x_3,x_{10},x_{11}\}$ for some elements $x_{10}$ and $x_{11}$ in $E(M) - \{x_1,x_2,\dots,x_9\}$.

    \begin{sublemma}\label{no4}
        $M$ has no $4$-circuit containing $\{x_2,x_3\}$.
    \end{sublemma}

    Assume $M$ has such a circuit $C$. As $\{x_1,x_2,x_4,x_5\}$ is a cocircuit, $C$ contains $x_4$ or $x_5$. By Lemma~\ref{4CocctInd}, $C$ does not contain $\{x_4,x_5\}$ otherwise $r(\{x_1,x_2,x_3,x_4,x_5,x_6\}) = 3$ and $M$ has $\{x_1,x_2,x_4,x_5\}$ as a circuit and a cocircuit. Thus $C$ is $\{x_2,x_3,x_4,\alpha\}$ or $\{x_2,x_3,x_5,\beta\}$ for some elements $\alpha$ and $\beta$. Suppose $C = \{x_2,x_3,x_5,\beta\}$. Then $\beta \in \{x_6,x_8,x_9\}$ and we obtain a contradiction to orthogonality with the cocircuit $\{x_1,x_3,x_{10},x_{11}\}$. Thus $C = \{x_2,x_3,x_4,\alpha\}$ and, by symmetry, we may assume that $\alpha = x_{10}$.

    As $\{x_2,x_3,x_4,x_{10}\}$ and $\{x_1,x_2,x_3\}$ are circuits of $M$, we can eliminate $x_2$ to obtain that $\{x_1,x_3,x_4,x_{10}\}$ is a circuit of $M$ since each of $x_1,x_3,$ and $x_4$ is in just one triangle. But this circuit meets the cocircuit $\{x_2,x_3,x_7,x_8\}$ in a single element, a contradiction. Hence \ref{no4} holds.

    Consider $M /x_2,x_3\backslash x_1$. By \ref{no4} and the fact that each of $x_2$ and $x_3$ is in a single triangle, we deduce that this matroid is simple. Now $M / x_2,x_3 \backslash x_1$ has no triad $T^*$ otherwise $M$ has $T^* \cup x_1$ as a cocircuit that meets $\{x_1,x_2,x_3\}$ in a single element. We deduce that $M / x_2,x_3 \backslash x_1$ contradicts the minimality of $M$.
\end{proof}

\begin{lemma}\label{ring}
     Let $M$ be a minor-minimal ternary matroid in $\mathcal{M}_4$. Assume that no element of $M$ is more than one triangle. Then there is no integer $n$ such that $M$ has $T_1,T_2,\dots,T_n$ as disjoint triangles and has $C_1^*,C_2^*,\dots,C_n^*$ as $4$-cocircuits such that, interpreting all subscripts modulo $n$, the sets $C^*_i \cap T_i$ and $C^*_{i-1} \cap T_i$ are distinct and $|C_i^* \cap T_i| = 2 = |C^*_i \cap T_{i+1}|$ for all $i$.
\end{lemma}

\begin{proof}
    Assume that there is such an integer. Choose a least such integer $n$. By Lemmas \ref{4CocctInd} and \ref{3tri}, $n \geq 4$. Let $T_i = \{x_{i1}, x_{i2}, x_{i3}\}$ for all $i$. We can shuffle the labels within each $T_i$ such that, for $j < n$, the cocircuit $C_j^*$ is $\{x_{j2},x_{j3}, x_{(j+1)2}, x_{(j+1)3}\}$ when $j$ is odd and $\{x_{j1},x_{j2}, x_{(j+1)1}, x_{(j+1)2}\}$ when $j$ is even. Moreover, we can take $C_n^*$ to be $\{x_{n2}, x_{n3},x_{11},x_{12}\}$ when $n$ is odd and $\{x_{n1}, x_{n2}, x_{11}, x_{12}\}$ when $n$ is even. Now consider $M/x_{11},x_{13} \backslash x_{12}$. Suppose it has a triad $T^*$. Then $T^* \cup x_{12}$ is a cocircuit of $M$ that meets $T_1$ in a single element, a contradiction. By the minimality of $M$, it has a $4$-circuit $C$ that contains $\{x_{11}, x_{13}\}$. By orthogonality, $C$ contains $x_{22}$ or $x_{23}$ and $C$ also contains a member of $\{x_{n2}, x_{n3}\}$ when $n$ is odd and a member of $\{x_{n1},x_{n2}\}$ when $n$ is even. The cocircuit $\{x_{21}, x_{22},x_{31},x_{32}\}$ implies that $x_{22} \not\in C$, so $x_{23} \in C$. Thus $C = \{x_{11},x_{13},x_{23},x_{n\alpha}\}$ for some $\alpha$ in $\{1,2,3\}$. Now, by Lemma~\ref{4co}, $M$ has a $4$-cocircuit $D^*$ that contains $\{x_{21},x_{23}\}$. By Lemma~\ref{4CocctInd}, this cocircuit does not meet $T_1$, Thus $x_{n\alpha} \in D^*$ so $|D^* \cap T_n| = 2$. The minimality of $n$ is contradicted unless $D^* \cap T_n = T_n \cap C^*_{n-1}$. In the exceptional case, $D^* \Delta C_{n-1}^*$ is a $4$-cocircuit $D^*_{n-1}$ of $M$ that meets $T_{n-1}$ in exactly two elements. Then the triangles $T_2,T_3,\dots,T_{n-2},$ and $T_{n-1}$ and the cocircuits $C_2^*,C_3^*,\dots,C_{n-2}^*,$ and $D^*_{n-1}$ violate the minimality of $n$. 
\end{proof}

\begin{lemma}\label{no222}
    Let $M$ be a minor-minimal ternary matroid in $\mathcal{M}_4$. Suppose that $M$ has $\{x_1,x_2,x_4,x_5\}$ as a cocircuit and has $\{x_1,x_2,x_3\}, \{x_1,x_4,x_6\}$, and $\{x_2,x_4,x_7\}$ as triangles. Then $M$ is isomorphic to $F_7^-,P_7,$ or $M(K_5)$.
\end{lemma}

\begin{proof}
    Assume that $M$ is not isomorphic to $F_7^-,P_7$ or $M(K_5)$. Since $M$ has a triangle containing $x_5$, by orthogonality, this triangle contains $x_1,x_2$, or $x_4$. Thus one of those elements is in at least three triangles, a contradiction to Lemma~\ref{NoThreeTriangles}.
\end{proof}
 
\begin{figure}[t]
    \centering
\[
\begin{blockarray}{ccccccccccc}
 & b_{r+1} & a_1 & a_2 & a_3 & \dots  &a_{r-1} & a_r & a_{r+1}\\
\begin{block}{c(cccccccccc)}
  b_1     & 1 & 1 & 0 & 0 & & 0 & 1 & 0\\
  b_2     & 1 & 1 & 1 & 0 & & 0 & 1 & 1\\
  b_3     & 1 & 0 & 1 & 1 & & 0 & 1 & 1\\
  b_4     & 1 & 0 & 0 & 1 & & 0 & 1 & 1\\
  \vdots  &   &   &   &   & &   &   &  \\
  b_{r-1} & 1 & 0 & 0 & 0 & & 1 & 1 & 1\\
  b_r     & 1 & 0 & 0 & 0 & & 1 & 0 & 1\\
\end{block}
\end{blockarray}
 \]
    \caption{The matrix $A_r$.}
    \label{Ar}
\end{figure}

For $r \geq 4$, let $M_r$ be the matroid that is represented over $GF(3)$ by the matrix $[I_r|A_r]$ where $A_r$ is the ternary matrix shown in Figure \ref{Ar}. We omit the routine proof of the following result.

\begin{lemma}\label{mrsub1}
    For all $r \geq 6$, $$M_r / a_1,a_2 \backslash b_1,b_2 \cong M_{r-2}.$$
\end{lemma}

For all $i$ in $[r+1]$, we see that $\{b_i,a_i,b_{i+1}\}$ is a triangle of $M_r$, and $\{a_i,b_{i+1}, b_{i+2}, a_{i+2}\}$ is a cocircuit where all subscripts are interpreted modulo $r+1$.

\begin{figure}
    \centering
    \begin{tikzpicture}[tight/.style={inner sep=1pt}, loose/.style={inner sep=.7em}, scale=0.55]
        \coordinate (bl) at (3.5,0);
        \draw[fill=black] (bl) circle (6pt);
        
        \coordinate (br) at (8.5,0);
        \draw[fill=black] (br) circle (6pt);
        
        \coordinate (ml) at (2,4);
        \draw[fill=black] (ml) circle (6pt);
        
        \coordinate (mr) at (10,4);
        \draw[fill=black] (mr) circle (6pt);
        \coordinate (top) at (6,7);
        
        \draw[fill=black] (top) circle (6pt);
    
    \draw (ml) edge["$b_1$", tight]  (top)
          (top) edge["$b_2$", tight] (mr)
          (mr) edge["$b_3$", tight] (br)
          (br) edge["$b_4$", tight] (bl)
          (bl) edge["$b_5$", tight] (ml);

    \draw (br) edge["$a_3$", tight] (top)
          (top) edge["$a_5$", tight] (bl);
    \draw (ml) edge["$a_1$", tight] (mr);
    \draw (br) edge["$a_4$", tight] (ml);
    \draw (mr) edge["$a_2$", tight] (bl);
        
    \end{tikzpicture}
    \caption{$K_5.$}
    \label{fig:new}
\end{figure}

\begin{lemma}\label{mrsub2}
$M_4 \cong M(K_5)$ and $M_5 \cong M(K_{2,2,2})$.
\end{lemma}

\begin{proof}
 It is straightforward to check that each of the triangles in the graph $K_5$ in Figure \ref{fig:new} is a circuit in $M_4$. It follows that $M_4/e$ is binary for all $e$. Thus $M_4$ is binary as $r(M_4) = 4$. Therefore, as $M_4$ is ternary,  $M_4$ is regular.  Since $|E(M_4)| = 10 =$ $ \binom{r(M_{4})+1}{2}$, it follows by a result of Heller [\ref{heller}] that $M_4 \cong M(K_5)$. It is straightforward to check that the triangles of $M_5$ have the structure of Figure \ref{fig:k222}. Thus, by Lemma \ref{tiedown}, $M_5 \cong M(K_{2,2,2})$.
\end{proof}

\begin{lemma}\label{new-1}
    Let $M$ be a minor-minimal ternary matroid in $\mathcal{M}_4$. Assume that $M$ is not isomorphic to $M(K_5)$ or $M(K_{2,2,2})$. For some $k \geq 4$, let $y_1,z_1,y_2,z_2,\dots,y_k,z_k,y_{k+1}$ be a sequence of distinct elements of $M$ such that $\{y_i,z_i,y_{i+1}\}$ is a triangle for all $i$ in $[k]$ and $\{z_j,y_{j+1},y_{j+2},z_{j+2}\}$ is a cocircuit for all $j$ in $[k-2]$. Assume that none of $z_1,z_2,\dots,z_{k-1}, $ or $z_k$ is in more than one triangle of $M$. Then $M$ has elements $z_{k+1}$ and $y_{k+2}$ that are not in $\{y_1,z_1,y_2,\dots,y_k,z_k,y_{k+1}\}$ such that $\{y_{k+1}, z_{k+1}, y_{k+2}\}$  is a triangle, $\{z_{k-1},y_k, y_{k+1},z_{k+1}\}$ is a cocircuit, and $z_{k+1}$ is in only one triangle of $M$.
\end{lemma}

\begin{proof}
    As $M/ z_k \backslash y_{k+1}$ is simple, it has a triad $T^*_{k+1}$. Then $T^*_{k+1} \cup y_{k+1}$ is a cocircuit of $M$ and $y_k \in T^*_{k+1}$. By orthogonality, $z_{k-1}$ or $y_{k-1}$ is in $T^*_{k+1}$. Suppose $y_{k-1} \in T^*_{k+1} $. Then $z_{k-2}$ or $y_{k-2}$ is in $T^*_{k+1}$. In the latter case, the triangle $\{y_{k-3},z_{k-3},y_{k-2}\}$ contradicts orthogonality. In the former case, $M^*|\{z_{k-2},y_{k-1},y_k,z_k,y_{k+1}\} \cong U_{3,5}$, a contradiction. Then $y_{k-1} \not\in T^*_{k+1}$, so $z_{k-1} \in T^*_{k+1}$. Let $T^*_{k+1} - \{z_{k-1},y_k\} = \{z_{k+1}\}$. By orthogonality and Lemma~\ref{No4PtLine},  $z_{k+1} \not\in \{y_1,z_1,y_2,\dots,z_k,y_{k+1}\}$. Now $z_{k+1}$ is in a triangle $T$ of $M$. By Lemma~\ref{NoThreeTriangles}, $y_k$ is not in three triangles, and, by assumption, $z_{k-1}$ is in only one triangle, so $z_{k+1}$ is in only one triangle and this triangle avoids $\{z_{k-1},y_k\}$. Thus $y_{k+1} \in T$. Let $T - \{y_{k+1},z_{k+1}\} = \{y_{k+2}\}$. The known $4$-cocircuits of $M$ imply that $y_{k+2} \not\in \{y_1,z_1,y_2,z_2,\dots,y_{k+1},z_{k+1}\}$ unless $y_{k+2} = y_1$.

    \begin{sublemma}\label{new-1sub} If $y_{k+2} = y_1$, then $M$ is isomorphic to $M(K_5)$ or $M(K_{2,2,2})$.
    \end{sublemma}

    Assume $y_{k+2} = y_1$. Now $M / z_{k+1} \backslash y_{k+2}$ is simple and so has triad $T^*_{k+2}$. Then $T^*_{k+2} \cup y_{k+2}$ is a cocircuit of $M$ that contains $y_{k+1}$. Thus $z_k$ or $y_k$ is in $T^*_{k+2}$, and $z_1$ or $y_2$ is in $T^*_{k+2}$. The triangles $\{y_{k-1},z_{k-1},y_k\}$ and $\{y_2,z_2,y_3\}$ imply that $z_k$ and $z_1$ are in $T^*_{k+2}$, so $\{z_k,y_{k+1},y_1,z_1\}$ is a cocircuit of $M$.

    Let $V = \{y_1,z_1,\dots,y_{k+1},z_{k+1}\}$. Then $V$ is spanned by $\{y_1,y_2,\dots,y_{k+1}\}$.
    Moreover, we get   $\{z_1,y_2,y_3,z_3\}, \{z_2,y_3,y_4,z_4\},\dots,  \{z_{k-1},y_k,y_{k+1},z_{k+1}\},$\linebreak and
    $ \{z_k,y_{k+1},y_1,z_1\}$ as cocircuits of $V$, where each of which contains an element that is not in the union of its predecessors, so $r^*(V) \leq 2(k+1) - k$. Since $r(V) \leq k+1$, it follows that $\lambda(V) \leq 1$. Thus $|E(M) - V| \leq 1$. Assume $E(M) - V = \{w\}$. Then $w$ is in a triangle. But each of $y_1,y_2,\dots,y_{k+1}$ is already in two triangles, while each of $z_1,z_2,\dots,z_{k+1}$ is already in one triangle. Thus $w$ is not in a triangle, a contradiction. We deduce that $E(M) - V = \emptyset$, so $\lambda(V) = 0$. Hence $0 = r(V) + r^*(V) - 2(k+1)$. If $r(V) = k+1$, then $E(M) - \cl(\{y_1,y_2,\dots,y_k\}) \subseteq \{z_k,y_{k+1},z_{k+1}\}$, so $M$ has a triad, a contradiction. Hence $r(V) \leq k$. The $k$ cocircuits listed above imply that $r(V) \geq k$. Hence $r(V) = k = r(M)$. 

    We now construct a representation of $M$. Let $B \cup y_{k+1} = \{y_1,y_2,\dots,y_{k+1}\}$ where $B$ is the basis $\{y_1,y_2,\dots,y_k\}$. Then $B \cup y_{k+1}$ contains a circuit containing $y_{k+1}$. By orthogonality with the known $4$-cocircuits of $M$, we deduce that $\{y_1,y_2,\dots,y_{k+1}\}$ is a circuit of $M$. 

    \begin{figure}
    \centering
\[
\begin{blockarray}{ccccccccccc}
 & y_{k+1} & z_1 & z_2 & z_3 & \dots  &z_{k-1} & z_k & z_{k+1}\\
\begin{block}{c(cccccccccc)}
  y_1     & 1 & 1 & 0 & 0 & & 0 & 1 & w_{k+1}\\
  y_2     & 1 & v_2 & 1 & 0 & & 0 & 1 & 1\\
  y_3     & 1 & 0 & v_3 & 1 & & 0 & 1 & 1\\
  y_4     & 1 & 0 & 0 & v_4 & & 0 & 1 & 1\\
  \vdots  &   &   &   &   & &   &   &  \\
  y_{k-1} & 1 & 0 & 0 & 0 & & 1 & 1 & 1\\
  y_k     & 1 & 0 & 0 & 0 & & v_k & w_k & 1\\
\end{block}
\end{blockarray}
 \]
    \caption{A ternary representation for $M$.}
    \label{PartialRepn}
\end{figure}

From this representation, we now make it a ternary representation for $M$. By using the fundamental circuits with respect to $B$ and scaling the rows so that $y_{k+1}$ consists of all ones, we see that the matrix in Figure~\ref{PartialRepn} is a representation for $M$ where $v_2,v_3,\dots,v_k$ are all non-zero. The circuits $\{y_k,z_k,y_{k+1}\}$ and $\{y_{k+1},z_{k+1},y_1\}$ imply we may assume that the columns $z_k$ and $z_{k+1}$ are as shown, where $w_k$ and $w_{k+1}$ may be zero. Now, the cocircuits $\{z_{k-1},y_k,y_{k+1},z_{k+1}\}$ and $\{z_k,y_{k+1},y_1,z_1\}$ imply that $w_k=0~=w_{k+1}$. The cocircuit $\{z_1,y_2,y_3,z_3\}$ implies $v_3 = 1$. By symmetry, each of $v_4,v_5,\dots,v_k$ is $1$. Finally, the cocircuit $\{z_{k+1},y_1,y_2,z_2\}$ implies that $v_2 = 1$. Hence $M$ is represented as shown in Figure~\ref{repn4}.

By Lemma~\ref{mrsub2}, if $k$ is $4$ or $5$, then $M$ is isomorphic to $M(K_5)$ or $M(K_{2,2,2})$, respectively. By Lemma~\ref{mrsub1}, if $k \geq 6$, then $M$ has $M(K_5)$ or $M(K_{2,2,2})$ as a proper minor. The minimality of $M$ implies that $\ref{new-1sub}$ holds. Hence so does the lemma.
\end{proof}

    \begin{figure}
    \centering
\[
\begin{blockarray}{ccccccccccc}
 & y_{k+1} & z_1 & z_2 & z_3 & \dots  &z_{k-1} & z_k & z_{k+1}\\
\begin{block}{c(cccccccccc)}
  y_1     & 1 & 1 & 0 & 0 & & 0 & 1 & 0\\
  y_2     & 1 & 1 & 1 & 0 & & 0 & 1 & 1\\
  y_3     & 1 & 0 & 1 & 1 & & 0 & 1 & 1\\
  y_4     & 1 & 0 & 0 & 1 & & 0 & 1 & 1\\
  \vdots  &   &   &   &   & &   &   &  \\
  y_{k-1} & 1 & 0 & 0 & 0 & & 1 & 1 & 1\\
  y_k     & 1 & 0 & 0 & 0 & & 1 & 0 & 1\\
\end{block}
\end{blockarray}
 \]
    \caption{The matrix $A_r$ in Lemma \ref{new-1}.}
    \label{repn4}
\end{figure}   

\begin{lemma}\label{new3}
    Let $M$ be a $3$-connected minor-minimal ternary matroid in $\mathcal{M}_4$. Assume that $M$ has $\{x_1,x_2,x_4,x_5\}$ as a cocircuit and has $\{x_1,x_2,x_3\}, \linebreak\{x_1,x_4,x_6\}$, and $\{x_2,x_5,x_7\}$ as triangles, but $M$ has no triangle containing $\{x_4,x_5\}$. Then $M$ is isomorphic to $P_7, M^*(K_{3,3}), Q_9, M(K_5), $ or $M(K_{2,2,2})$.
\end{lemma}

\begin{proof}
     Assume that $M$ is not isomorphic to $P_7, Q_9, M^*(K_{3,3})$, or $M(K_{2,2,2})$. By Lemma~\ref{4CocctInd}, $r(\{x_1,x_2,x_4,x_5\}) = 4$. Moreover, by Lemma~\ref{No4PtLine}, $M$ has no $4$-point line. Thus $x_7 \not = x_3$. Also $x_7 \not= x_6$ otherwise $r(\{x_1,x_2,x_4,x_5\}) = 3$, a contradiction. Thus $x_1,x_2,\dots,x_6, $ and $x_7$ are distinct. 

    By Lemma \ref{no222}, $M/x_4\backslash x_6$ is simple. By the minimality of $M$, the matroid $M/x_4 \backslash x_6$ has a triad $T_6^*$, and $T^*_6\cup x_6$ is a cocircuit of $M$ that avoids $x_4$. Thus $M$ has $\{x_1,x_2,x_5,x_6\}$, $\{x_1,x_2,x_6,x_7\}$, or $\{x_1,x_3,x_6,x_8\}$ as a cocircuit for some element $x_8$ not in $\{x_1,x_2,\dots,x_7\}$. The first two cases violate Lemma~\ref{C1C2_intersection_not_2}. Thus $M$ has $\{x_1,x_3,x_6,x_8\}$ as a cocircuit. Now $x_8$ is in a triangle $T$ of $M$. By Lemma~\ref{NoThreeTriangles}, $x_1 \not\in T$. Thus $x_3$ or $x_6$ is in $T$.

    First assume that $\{x_3,x_8\} \subseteq T$ and let $\alpha$ be the third element of $T$. We show next that $\alpha \not\in \{x_1,x_2,\dots,x_8\}$. Assume $\alpha \in \{x_1,x_2,\dots,x_8\}$. Then \linebreak $\alpha \in \{x_6,x_7\}$. By Lemma~\ref{4CoCctNoTriangle}, $\alpha \not= x_6$. Thus $\alpha = x_7$. Now $x_7$ is in a $4$-cocircuit $D^*$. By Lemma~\ref{no222}, $\{x_2,x_3,x_7\}\not\subseteq D^* $. By orthogonality, $x_2$ or $x_5$ is in $D^*$, and $x_3$ or $x_8$ is in $D^*$. Moreover, $x_1\not\in D^*$. Thus if $x_2$ or $x_3$ is in $D^*,$ then both are, a contradiction. Hence $\{x_5,x_7,x_8\} \subseteq D^*$. Let $D^* = \{x_5,x_7,x_8,\beta\}$. The triangles $\{x_1,x_2,x_3\}$ and $\{x_1,x_4,x_6\}$ imply that $\beta \not\in \{x_1,x_2,\dots,x_8\}$. Now $\beta$ is in a triangle $T'$. By Lemma~\ref{NoThreeTriangles} and the lemma's hypothesis, $x_7 \not\in T'$ and $x_5 \not\in T'$. Thus $x_8 \in T'$. Let $T' = \{\beta ,x_8,\gamma\}$. The cocircuit $\{x_1,x_3,x_6,x_8\}$ implies, by Lemma~\ref{NoThreeTriangles}, that $\gamma =~ x_6$, that is, $\{\beta ,x_6,x_8\}$ is a triangle. Let $X = \{x_1,x_2,\dots, x_8,\beta \}$. Then $r(X) = 4$ and $\lambda(X) \leq 4 + (9-3) - 9  = 1$, so $|E(M) - X| \leq 1$. As $r(\{x_1,x_2,x_3,x_5,x_7,x_8\})~ =~3$, we deduce that $|E(M) - X| = 1$ otherwise $M$ has a triad. Let $E(M) - X = \{\delta\}$. Deleting the cocircuits $\{x_1,x_2,x_4,x_5\}$ and $\{x_1,x_3,x_6.x_8\}$ from $M$ leaves $\{x_7,\beta,\delta\}$, which must be a triangle. Thus $x_7$ is in three triangles, a contradiction to Lemma~\ref{NoThreeTriangles}. We conclude that $\alpha \not\in\{x_1,x_2,\dots,x_8\}$.

    By the minimality of $M$, the simple matroid $M/x_5\backslash x_7$ has a triad $T^*_7$ and $T^*_7 \cup x_7$ is a cocircuit of $M$ that avoids $x_5$, so it contains $x_2$. Now $x_1 \not\in T^*_7$, otherwise $\{x_1,x_2,x_4,x_7\}$ or $\{x_1,x_2,x_6,x_7\}$ is a cocircuit of $M$. The first case gives the contradiction that $M^*$ has a $U_{3,5}$-minor, while the second case violates Lemma~\ref{C1C2_intersection_not_2}. We deduce that $\{x_2,x_3,x_7,x_8\}$ or $\{x_2,x_3,x_7,\alpha\}$ is a cocircuit. Suppose $\{x_2,x_3,x_7,x_8\}$ is a cocircuit. As $\{x_1,x_3,x_6,x_8\}$ is a cocircuit, eliminating $x_8$ gives a cocircuit contained in $\{x_1,x_2,x_3,x_6,x_7\}$. The triangle $\{x_3,x_8,\alpha\}$ implies that the cocircuit avoids $x_3$, so it must be $\{x_1,x_2,x_6,x_7\}$. This gives a contradiction to Lemma~\ref{C1C2_intersection_not_2}. We deduce that $\{x_2,x_3,x_7,\alpha\}$ is a cocircuit of $M$.

    Let $X = \{x_1,x_2,\dots,x_8,\alpha\}$. Then $r(X) \leq 5$ and $\lambda(X) \leq r(X) - 3$, so $\lambda(X) \leq 2$. Suppose $r(X) = 4$. Then $\lambda(X) \leq 1$, so $M$ has at most one element that is not in $X$. Assume that $E(M) - X = \{e\}$. Then $r(M) = 4$. As $M$ has $\{x_1,x_2,x_4,x_5\}, \{x_1,x_3,x_6,x_8\},$ and $\{x_2,x_3,x_7,\alpha\}$ as cocircuits, by considering the set obtained from $E(M)$ by deleting each two of these three cocircuits, we deduce that $M$ has $\{x_6,x_8,e\}$, $\{x_7,\alpha, e\}$, and $\{x_4,x_5,e\}$ as triangles. Thus $e$ is in at least three triangles, a contradiction to Lemma~\ref{NoThreeTriangles}. We conclude that $X=  E(M)$. Then $\{x_1,x_2,x_3,x_8,\alpha\}$ is a hyperplane of $M$, so $M$ has $\{x_4,x_5,x_6,x_7\}$ as a cocircuit. As $\{x_4,x_5\}$ is not in a triangle, we obtain the contradiction that $M \cong Q_9$. We conclude that $r(X) > 4$, so $r(X) = 5$.
    
     Assume that $r(M) = 5.$ If $\lambda(X) \leq 1$, then $M$ has at most one element not in $X$. As $r(\{x_1,x_2,\dots,x_7\}) = 4$, it follows that $M$ has a cocircuit of size less than four. Then $\lambda(X) = 2$, so $r(X) = 5$ and $r(E(M) - X) = 2$. Hence $E(M) - X$ is a line containing exactly two or exactly three points. Assume $E(M) - X = \{y_1,y_2\}$. Then each $y_i$ is in a triangle $T_i$. Because $X$ is a union of cocircuits, $\{y_1,y_2\}$ is not contained in a triangle. By Lemma~\ref{NoThreeTriangles}, none of $x_1,x_2,$ or $x_3$ is in $T_i$. By assumption, $T_i$ avoids $\{x_4,x_5\}$. Thus $T_i - y_i \subseteq \{x_6,x_7,x_8,\alpha\}$ for each $i$. By Lemma~\ref{No4PtLine}, $M$ has no line with more than three points, so, by using the cocircuits $\{x_1,x_3,x_6,x_8\}$ and $\{x_2,x_3,x_7,\alpha\}$, we may assume that $\{y_1,x_6,x_8\}$ and $\{y_2,x_7,\alpha\}$ are triangles. Then $r(\{x_1,x_2,x_3,x_4,x_6,x_8,y_1,\alpha\}) \leq 4$ so $M$ has a cocircuit of size less than four, a contradiction. We conclude that $E(M) - X$ is a triangle when $r(M) =5$. 

    Assume $r(M) > 5$. Then, as $r(X) = 5$, we must have $\lambda(X) = 2$ otherwise $M$ has a cocircuit of size less than four. Now, $M$ has a triangle $Y$ disjoint from $\cl(X)$. By Theorem \ref{GGWTutte}, $M$ has a $12$-element rank-$5$ minor $N$ having ground set $X \cup Y$ with $M|X = N|X$ and $M|Y = N|Y$ such that $\lambda_N(Y) = 2$. Then $r(N) = 5$. Now  $E(N)$ can be written as a disjoint union of triangles. Thus $N$ has no triads. By the minimality of $M$, we see that $N$ must have a $2$-element cocircuit $S^*$. As $r(X) = 5$, and $Y$ is a triangle of $N$, we see that $S^* \subseteq X$. The triangles in $X$ imply that $S^*$ is $\{x_4,x_6\}$,$\{x_5,x_7\}$ or $\{x_8,\alpha\}$. Since $r_N(Y) = 2$, each of the cocircuits $\{x_1,x_2,x_4,x_5\}, \{x_1,x_3,x_6,x_8\}$ and $\{x_2,x_3,x_7,\alpha\}$ of $M$ is also a cocircuit of $N$. The additional cocircuit $S^*$ gives a contradiction. We conclude that $r(M) = 5$. 

    Let $E(M) - X = \{e,f,g\}$, which is a triangle of $M$. The cocircuits of $M$ imply that $M$ has $\{e,f,g,x_7,\alpha\}, \{e,f,g,x_4,x_5\}$ and $\{e,f,g,x_6,x_8\}$ as planes. We may view $M$ as a restriction of a rank-$5$ ternary projective space $P$. Let $h$ be the unique point in $\cl_P(\{e,f,g\}) - \{e,f,g\}$. Since $\{x_4,x_5\}$ is not in a triangle of $M$, we deduce that $\{x_4,x_5,h\}$ is a triangle of $P$. If $\{x_4,x_5,x_7,\alpha\},\{x_4,x_5,x_6,x_8\}$, or $\{x_6,x_7,x_8,\alpha\}$ is a circuit of $M$, we obtain the contradiction that $r(X) = 4$. Thus each of $\cl_P(\{x_7,\alpha\}), \cl_P(\{x_4,x_5\})$, and $\cl_P(\{x_6,x_8\})$ meets $\{e,f,g,h\}$ in a distinct point. Hence we may assume that $\{x_6,x_8,e\}$ and $\{x_7,\alpha,f\}$ are triangles of $M$. 

    Next we construct a ternary representation for the matroid $M'$ that is obtained from $M$ by adjoining the element $h$. Then $\{x_1,x_2,x_4,x_5,x_8\}$ is a basis $B$ for this matrix and the representation is shown in Figure \ref{new3repn}.
\begin{figure}
    \centering
\[
\begin{blockarray}{ccccccccccc}
 & x_3 & x_6 & x_7 & \alpha & e  & f & g & h\\
\begin{block}{c(cccccccccc)}
  x_1     & 1 & 1 & 0 & 1 & 1   & 1   & 1   & 0\\
  x_2     & 1 & 0 & 1 & 1 & 0   & 0   & 0   & 0\\
  x_4     & 0 & 1 & 0 & 0 & 1   & 0   & u_5 & 1\\
  x_5     & 0 & 0 & 1 & 0 & 0   & u_3 & u_6 & u_1\\
  x_8     & 0 & 0 & 0 & 1 & u_2 & u_4 & u_7 & 0\\
\end{block}
\end{blockarray}
 \]
 \caption{Building a representation for $M'$.}
    \label{new3repn}
\end{figure}
To verify that this is indeed a representation, we use the fundamental circuits of $x_3,x_6$, and $x_7$ with respect to $B$ and the circuit $\{x_3,x_8,\alpha\}$. Scaling rows $2$--$5$ so that their first non-zero entries are $1$, we deduce that the columns corresponding to $x_3,x_6,x_7,$ and $\alpha$ are as indicated. Moreover, the column corresponding to $h$ is as indicated, where $u_1 \not= 0$. The circuit $\{x_6,x_8,e\}$ implies that the column $e$ is as indicated with $u_2 \not= 0$. Because $\{e,f,g,h\}$ is a $4$-point line, the second coordinates of columns $f$ and $g$ are zero. The triangle $\{x_7,\alpha,f\}$ implies that the third coordinate of column $f$ is zero while we may take the first coordinate of $f$ to be $1$. The fourth and fifth coordinates of column $f$ are $u_3$ and $u_4$ both of which are non-zero. The triangle $\{e,g,h\}$ implies that we may take the first coordinate of column $g$ to be $1$. The last three coordinates of $g$ are non-zero, $u_5,u_6,$ and $u_7$, respectively. 

 The circuit $\{x_7,\alpha,f\}$ implies that $u_3 = -1$ and $u_4 = 1$. The circuit $\{e,f,g\}$ implies that $u_5 =-1, u_6 = 1,$ and $u_7 = -u_2-1$. Finally, the circuit $\{f,g,h\}$ implies that $u_1 = 1$ and $u_2 = 1$, so Figure \ref{new3repn2} is a representation for the extension of $M$ by the element $h$. One can now check that  $M' / g \backslash e,f,h~\cong~Q_9$, a contradiction. We conclude that $\{x_3,x_8\} \not\subseteq T$.

        \begin{figure}[b]
    \centering
\[
\begin{blockarray}{ccccccccccc}
 & x_3 & x_6 & x_7 & \alpha & e  & f & g & h\\
\begin{block}{c(cccccccccc)}
  x_1     & 1 & 1 & 0 & 1 & 1 &  1 & 1  & 0\\
  x_2     & 1 & 0 & 1 & 1 & 0 &  0 & 0  & 0\\
  x_4     & 0 & 1 & 0 & 0 & 1 &  0 & -1 & 1\\
  x_5     & 0 & 0 & 1 & 0 & 0 & -1 &  1 & 1\\
  x_8     & 0 & 0 & 0 & 1 & 1 &  1 &  1 & 0\\
\end{block}
\end{blockarray}
 \]
 \caption{A matrix $A$ such that $M' \cong M[I_5|A]$.}
    \label{new3repn2}
\end{figure}

    We now know that $\{x_6,x_8\} \subseteq T$. By Lemma~\ref{no222} applied to the cocircuit $\{x_1,x_3,x_6,x_8\}$, we deduce that $x_3$ is not in a triangle apart from $\{x_1,x_2,x_3\}$, otherwise this triangle contains $x_8$, which we eliminated above. \linebreak Let $T = \{x_6,x_8,\alpha_1\}$. Then, by Lemma~\ref{no222}, the cocircuit $\{x_1,x_3,x_6,x_8\}$ implies that $M$ has no additional triangles containing $x_8$. Observe that $\alpha_1 \not\in \{x_1,x_2,\dots,x_8\}$ unless $\alpha_1 = x_7$. Now $M/x_8 \backslash \alpha_1$ is simple and so has a triad $T^*_1$ avoiding $x_8$. Then $x_6 \in T^*_1$, and $T^*_1 \cup \alpha_1$ is a cocircuit of $M$. By orthogonality and Lemma~\ref{4CoCctNoTriangle}, exactly one $x_1$ and $x_4$ is in $T^*_1$. Suppose $x_1 \in T^*_1$. Then either $\{x_1,x_3,x_6,\alpha_1\}$ is a cocircuit, or $\alpha_1 = x_7$ and $\{x_1,x_2,x_6,x_7\}$ is a cocircuit. In the first case, $M^*|\{x_1,x_3,x_6,x_8,\alpha_1\} \cong U_{3,5}$, a contradiction. In the second case, the union of the two triangles $\{x_1,x_2,x_3\}$ and $\{x_6,x_7,x_8\}$ contains two $4$-circuits, a contradiction to Lemma~\ref{C1C2_intersection_not_2}. We conclude that $x_1 \not\in T^*_1$. Thus $x_4 \in T^*_1$. Again if $\alpha_1 = x_7,$ then $T^*_1 \cup \alpha_1 = \{x_4,x_5,x_6,x_7\}$, and this cocircuit together with $\{x_1,x_2,x_4,x_5\}$ gives a contradiction to Lemma~\ref{C1C2_intersection_not_2}. Thus $\alpha_1 \not\in \{x_1,x_2,\dots,x_8\}$. Letting $$(x_7,x_5,x_2,x_3,x_1,x_4,x_6,x_8,\alpha_1) = (y_1,z_1,y_2,z_2,y_3,z_3,y_4,z_4,y_5)$$ in Lemma~\ref{new-1}, we deduce, by repeated applications of that lemma, that $M$ is isomorphic to $M(K_5)$ or $M(K_{2,2,2})$, otherwise $|E(M)|$ is infinite.
\end{proof}

\begin{lemma}\label{new4}
    Let $M$ be a minor-minimal ternary matroid in $\mathcal{M}_4$. If $M$ has $\{x_1,x_2,x_4,x_5\}$ as a cocircuit and has $\{x_1,x_2,x_3\},\{x_1,x_4,x_7\},$ $\{x_4,x_5,x_6\}$, and $\{x_2,x_5,x_8\}$ as triangles, then $x_1,x_2,\dots,x_7$, and  $x_8$ are distinct, and $M$ has a $4$-cocircuit $\{x_2,x_3,x_8,\sigma\}$ or $\{x_4,x_6,x_7,\tau\}$ where neither $\sigma$ nor $\tau$ is in $\{x_1,x_2,\dots,x_8\}$.
\end{lemma}

\begin{proof}
    By Lemma~\ref{4CocctInd}, $x_3 \not=x_6$ otherwise $r(\{x_1,x_2,x_4,x_5\}) = 3$. By symmetry, $x_7 \not= x_8$. Similarly, $x_3 \not= x_8$ and, by Lemma~\ref{No4PtLine}, $x_3 \not= x_7$. Thus $x_1,x_2,\dots,x_7$, and $x_8$ are distinct. 

    Consider $M/ x_5\backslash x_2,x_4$. This matroid is simple. By Lemma~\ref{C1C2_intersection_not_2}, $M$ does not have a $4$-cocircuit containing $\{x_2,x_4\}$ and avoiding $x_5$. Thus $M/ x_5 \backslash x_2,x_4$ is cosimple.  By the minimality of $M$, the matroid $M / x_5 \backslash x_2,x_4$ has a triad $T^*$. Thus one of $T^* \cup x_2, T^* \cup x_4,$ or $T^* \cup \{x_2,x_4\}$ is a cocircuit of $M$ where $x_5 \not\in T^*$.

    Assume that $T^* \cup \{x_2,x_4\}$ is a cocircuit of $M$. Then, by orthogonality. $\{x_6,x_8\} \subseteq T^*$ and also $x_1 \in T^*$. Thus $\{x_1,x_2,x_4,x_6,x_8\}$ is a cocircuit of $M$. Eliminating $x_1$ from the union of $\{x_1,x_2,x_4,x_5\}$ and $\{x_1,x_2,x_4,x_6,x_8\}$, we deduce that $M$ has a cocircuit $D^*$ that is contained in $\{x_2,x_4,x_5,x_6,x_8\}$. By orthogonality, $x_4 \not\in D^*$ and $x_2 \not\in D^*$, so $|D^*| \leq 3$, a contradiction. Hence $T^* \cup \{x_2,x_4\}$ is not a cocircuit of $M$. 

    Next suppose that $T^* \cup x_2$ is a cocircuit of $M$. As $x_5 \not\in T^*$, we deduce that $x_8 \in T^*$. Also $x_1$ or $x_3$ is in $T^*$. If $x_1 \in T^*$, we get a contradiction to Lemma~\ref{4CocctInd}. Thus $x_3 \in T^*$, so $T^* \cup x_2 = \{x_2,x_3,x_8,\sigma\}$ for some element $\sigma$ that, by orthogonality, is not in $\{x_1,x_2,\dots,x_8\}$. By symmetry, if $T^* \cup x_4$ is a cocircuit of $M$, then $T^* \cup x_4 = \{x_4,x_6,x_7,\tau\}$ for some element $\tau$ not in $\{x_1,x_2,\dots,x_8\}$.
\end{proof}

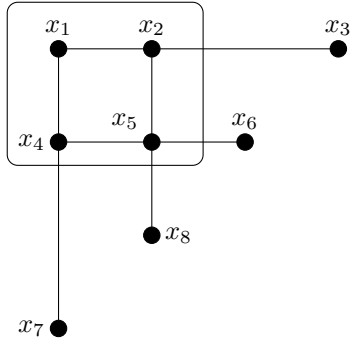
\begin{figure}
    \centering
    
  \begin{tikzpicture}[tight/.style={inner sep=1pt}, loose/.style={inner sep=.7em}, scale=0.62]

            	\coordinate (x7) at (0,0);
                \draw[fill=black] (x7) circle (5pt);
                \node at (x7)[label=left:\small $x_7$, tight]{};
                
				\coordinate (x4) at (0,4);
	            \draw[fill=black] (x4) circle (5pt);
                \node at (x4)[label=left:\small $x_{4}$, tight]{};
                
                \coordinate (x1) at (0,6);
                \draw[fill=black] (x1) circle (5pt);
                \node at (x1)[label=above :\small $x_{1}$, tight]{};                
				\coordinate (x2) at (2,6);
                \draw[fill=black] (x2) circle (5pt);
                \node at (x2)[label=above:\small $x_{2}$, tight]{};
                
				\coordinate (x3) at (6,6);
                \draw[fill=black] (x3) circle (5pt);
                \node at (x3)[label=above :\small $x_3$, tight]{};
            
                \coordinate (x5) at (2,4);
                \draw[fill=black] (x5) circle (5pt);
                \node at (x5)[label=above left:\small $x_5$, tight]{};            
                \coordinate (x6) at (4,4);
                \draw[fill=black] (x6) circle (5pt);
                \node at (x6)[label=above :\small $x_6$, tight]{};
                
                \coordinate (x8) at (2,2);
                \draw[fill=black] (x8) circle (5pt);
                \node at (x8)[label= right:\small $x_8$, tight]{};

                \draw (x1) -- (x7);
                \draw (x1) -- (x3);
                \draw (x4) -- (x6);
                \draw (x2) -- (x8);
  \draw[rounded corners] (-1.1, 7) rectangle (3.1,3.5);
			\end{tikzpicture}

    \caption{The triangles of $M$ at the beginning of the proof of Lemma~\ref{last}.}
    \label{fig:1245}
\end{figure}

\begin{lemma}\label{last}
    Let $M$ be a minor-minimal ternary matroid in $\mathcal{M}_4$. Assume that $M$ has $\{x_1,x_2,x_4,x_5\}$ as a cocircuit  and has $\{x_1,x_2,x_3\}$ and $ \{x_1,x_4,x_7\},$  as triangles. Then $M$ is isomorphic to $P_7,M^*(K_{3,3}), Q_9, \linebreak M(K_5),M(K_{2,2,2}),$ or $H_{12}$.
\end{lemma}

\begin{proof}

As $x_5$ is in a triangle of $M$, by Lemmas \ref{NoThreeTriangles} and \ref{new3}, we may assume that $\{x_4,x_5,x_6\}$ and $\{x_2,x_5,x_8\}$ are triangles of $M$. Moreover, by Lemma~\ref{new4}, $x_1,x_2,\dots,x_7$, and $x_8$ are distinct. We shall use the diagram in Figure \ref{fig:1245} to expose the symmetries that arise in the argument. The ring around the set $\{x_1,x_2,x_4,x_5\}$ is to indicate that this set is a cocircuit of $M$. By Lemma~\ref{new4}, $M$ has $\{x_2,x_3,x_8,y_2\}$ or $\{x_4,x_6,x_7,y_4\}$ as a cocircuit for some elements $y_2$ and $y_4$ not in $\{x_1,x_2,\dots,x_8\}$. Redrawing Figure \ref{fig:1245} as Figure \ref{fig:1245-2} corresponding to the two possibilities above, we see that these two cases are symmetric. 
We may assume that $M$ has $\{x_2,x_3,x_8,y_2\}$ as a cocircuit. By Lemma~\ref{new3}, $M$ has triangles $\{x_3,y_2,y_3\}$ and $\{x_8,y_2,y_8\}$.  As $M$ has no $4$-point lines, $y_3 \not=y_8$. Moreover, by orthogonality, $\{y_3,y_8\} \cap \{x_1,x_2,x_3,x_4,x_5,x_8,y_2\} = \emptyset.$

\begin{figure}[b]
    \centering
    \begin{subfigure}{.25\textwidth}
    
  \begin{tikzpicture}[tight/.style={inner sep=1pt}, loose/.style={inner sep=.7em}, scale=0.60]

            	\coordinate (x7) at (0,0);
                \draw[fill=black] (x7) circle (5pt);
                \node at (x7)[label=left:\small $x_7$, tight]{};
                
				\coordinate (x4) at (0,4);
	            \draw[fill=black] (x4) circle (5pt);
                \node at (x4)[label=left:\small $x_{1}$, tight]{};
                
                \coordinate (x1) at (0,6);
                \draw[fill=black] (x1) circle (5pt);
                \node at (x1)[label=above :\small $x_{4}$, tight]{};   
                
				\coordinate (x2) at (2,6);
                \draw[fill=black] (x2) circle (5pt);
                \node at (x2)[label=above:\small $x_{5}$, tight]{};
                
				\coordinate (x3) at (6,6);
                \draw[fill=black] (x3) circle (5pt);
                \node at (x3)[label=above :\small $x_6$, tight]{};
            
                \coordinate (x5) at (2,4);
                \draw[fill=black] (x5) circle (5pt);
                \node at (x5)[label=above left:\small $x_2$, tight]{};     
                
                \coordinate (x6) at (4,4);
                \draw[fill=black] (x6) circle (5pt);
                \node at (x6)[label=above :\small $x_3$, tight]{};

                \coordinate (y2) at (4,2);
                \draw[fill=black] (y2) circle (5pt);
                \node at (y2)[label= above:\small $y_2$, tight]{};             
                
                \coordinate (x8) at (2,2);
                \draw[fill=black] (x8) circle (5pt);
                \node at (x8)[label= right:\small $x_8$, tight]{};

                \draw (x1) -- (x7);
                \draw (x1) -- (x3);
                \draw (x4) -- (x6);
                \draw (x2) -- (x8);
  \draw[rounded corners] (-1.1, 7) rectangle (3.1,3.5);
  \draw[rounded corners] (.9,5) rectangle (5.1,1.3);
			\end{tikzpicture}
 
\end{subfigure}%
\hspace{1cm}
\begin{subfigure}{.5\textwidth}
  \centering
  \begin{tikzpicture}[tight/.style={inner sep=1pt}, loose/.style={inner sep=.7em}, scale=0.60]

            	\coordinate (x7) at (0,0);
                \draw[fill=black] (x7) circle (5pt);
                \node at (x7)[label=left:\small $x_8$, tight]{};
                
				\coordinate (x4) at (0,4);
	            \draw[fill=black] (x4) circle (5pt);
                \node at (x4)[label=left:\small $x_{5}$, tight]{};
                
                \coordinate (x1) at (0,6);
                \draw[fill=black] (x1) circle (5pt);
                \node at (x1)[label=above :\small $x_{2}$, tight]{};   
                
				\coordinate (x2) at (2,6);
                \draw[fill=black] (x2) circle (5pt);
                \node at (x2)[label=above:\small $x_{1}$, tight]{};
                
				\coordinate (x3) at (6,6);
                \draw[fill=black] (x3) circle (5pt);
                \node at (x3)[label=above :\small $x_3$, tight]{};
            
                \coordinate (x5) at (2,4);
                \draw[fill=black] (x5) circle (5pt);
                \node at (x5)[label=above left:\small $x_4$, tight]{};  
                
                \coordinate (x6) at (4,4);
                \draw[fill=black] (x6) circle (5pt);
                \node at (x6)[label=above :\small $x_6$, tight]{};
                
                \coordinate (x8) at (2,2);
                \draw[fill=black] (x8) circle (5pt);
                \node at (x8)[label= right:\small $x_7$, tight]{};

                \coordinate (y4) at (4,2);
                \draw[fill=black] (y4) circle (5pt);
                \node at (y4)[label= above:\small $y_4$, tight]{}; 

                \draw (x1) -- (x7);
                \draw (x1) -- (x3);
                \draw (x4) -- (x6);
                \draw (x2) -- (x8);
  \draw[rounded corners] (-1.1, 7) rectangle (3.1,3.5);
  \draw[rounded corners] (.9,5) rectangle (5.1,1.3);
			\end{tikzpicture}
\end{subfigure}
    \caption{The ringed sets correspond to cocircuits.}
    \label{fig:1245-2}
\end{figure}
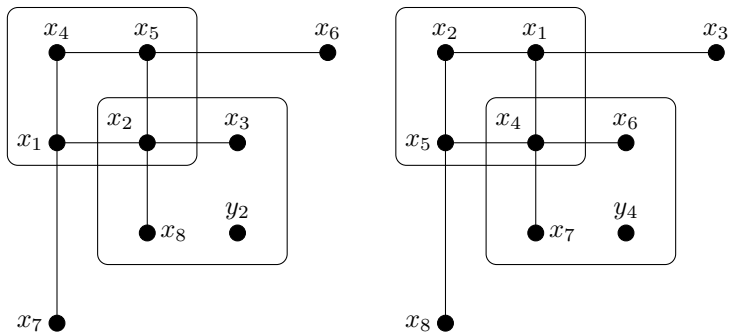

Suppose $x_6 = y_3$. There is a $4$-cocircuit $C_6^*$ containing $x_6$. Then, by Lemma \ref{4CoCctNoTriangle}, $C_6^*$ contains exactly one of $x_4$ and $x_5$ and exactly one of $x_3$ and $y_2$. By Lemma~\ref{C1C2_intersection_not_2}, $x_3 \not\in C^*_6$ and $x_5 \not\in C^*_6$. Thus $\{x_4,x_6,y_2\} \subseteq C^*_6$ so, by orthogonality, $C^*_6$ must contain $x_7$ and $y_8$, so $x_7 = y_8$. It follows by Lemma~\ref{lem33} that $M$ has $M^*(K_{3,3})$ as a minor, so we deduce that $x_6 \not= y_3$. By symmetry, $x_7\not=y_8$.

Next suppose that $x_6 = y_8$. Let $D^*_6$ be a $4$-cocircuit containing $x_6$. Then, by Lemma~\ref{4CoCctNoTriangle}, $D^*_6$ contains exactly one of $x_4$ and $x_5$ and contains exactly one of $x_8$ and $y_2$. Now $D^*_6$ cannot contain $\{x_4,x_6,x_8\}$ otherwise, by orthogonality, $|D^*_6| \geq 5$, a contradiction. Moreover, by Lemma~\ref{no222}, $\{x_5,x_6,x_8\} \not\subseteq D^*_6$. We deduce that $\{x_6,y_2\} \subseteq D^*_6$.

Suppose $\{x_4,x_6,y_2\} \subseteq D^*_6$. Then, by Lemma \ref{4CoCctNoTriangle}, $D^*_6$ contains $x_1$ or $x_7$ and contains $x_3$ or $y_3$. As $|D^*_6| = 4$, it follows that $y_3 = x_7$ and $D^*_6 = \{x_4,x_6,x_7,y_2\}$. The cocircuits $\{x_1,x_2,x_4,x_5\}, \{x_2,x_3,x_8,y_2\}$, and $\{x_4,x_6,x_7,y_2\}$ imply that $\lambda(\{x_1,x_2,\dots,x_8,y_2\}) \leq 4 + (9-3) - 9 = 1$, so $|E(M) - \{x_1,x_2,\dots,x_8,y_2\}| \leq 1$. As $r(\{x_2,x_4,x_5,x_6,x_8,y_2\}) = 3,$ to avoid $M$ having a cocircuit of size less than four, we must have that $E(M) - \{x_1,x_2,\dots,x_8,y_2\}$ contains a single element, say $\gamma$. The complement of the hyperplane $\{x_2,x_4,x_5,x_6,x_8,y_2\}$ is $\{x_1,x_3,x_7,\gamma\}$. As $\{x_1,x_2,x_3\}$, $\{x_3,x_7,y_2\}$, and $\{x_1,x_4,x_7\}$ are triangles, we get a contradiction to \linebreak Lemma~\ref{no222}. We conclude that $\{x_4,x_6,y_2\} \not\subseteq D^*_6$.

We now know that $\{x_5,x_6,y_2\} \subseteq D^*_6$. Then $x_2$ or $x_8$ is in $D^*_6$, and $x_3$ or $y_3$ is in $D^*_6$. Then, we obtain the contradiction that $|D^*_6| \geq5$ unless $x_3$ or $y_3$ is in $\{x_2,x_5,x_6,x_8,y_2\}$. Consider the exceptional case. As $x_1,x_2,\dots, x_8$, and $y_2$ are distinct, we must have that $y_3 \in \{x_2,x_5,x_6,x_8,y_2\}$. But we showed that $y_3 \not= x_6$. Also $y_3 \not= y_2$. By orthogonality between the triangle $\{x_3,y_2,y_3\}$ and the cocircuit $\{x_1,x_2,x_4,x_5\}$, we see that $y_3 \not\in \{x_2,x_5\}$. Finally, $y_3 \not= x_8$ or else the cocircuit $\{x_2,x_3,x_8,y_2\}$ contains a triangle, a contradiction to Lemma \ref{4CoCctNoTriangle}.  We conclude that $x_6 \not= y_8$. By symmetry, $x_7 \not = y_3$. Thus 
\begin{figure}
    \centering
    \begin{tikzpicture}[tight/.style={inner sep=1pt}, loose/.style={inner sep=.7em}, scale=0.60]

            	\coordinate (x7) at (0,0);
                \draw[fill=black] (x7) circle (5pt);
                \node at (x7)[label=left:\small $x_7$, tight]{};
                
				\coordinate (x4) at (0,4);
	            \draw[fill=black] (x4) circle (5pt);
                \node at (x4)[label=left:\small $x_{1}$, tight]{};
                
                \coordinate (x1) at (0,6);
                \draw[fill=black] (x1) circle (5pt);
                \node at (x1)[label=above :\small $x_{4}$, tight]{};   
                
				\coordinate (x2) at (2,6);
                \draw[fill=black] (x2) circle (5pt);
                \node at (x2)[label=above:\small $x_{5}$, tight]{};
                
				\coordinate (x3) at (6,6);
                \draw[fill=black] (x3) circle (5pt);
                \node at (x3)[label=above :\small $x_6$, tight]{};
            
                \coordinate (x5) at (2,4);
                \draw[fill=black] (x5) circle (5pt);
                \node at (x5)[label=above left:\small $x_2$, tight]{};            
                \coordinate (x6) at (4,4);
                \draw[fill=black] (x6) circle (5pt);
                \node at (x6)[label=above :\small $x_3$, tight]{};

                \coordinate (y2) at (4,2);
                \draw[fill=black] (y2) circle (5pt);
                \node at (y2)[label= above left:\small $y_2$, tight]{};             
                
                \coordinate (x8) at (2,2);
                \draw[fill=black] (x8) circle (5pt);
                \node at (x8)[label= left:\small $x_8$, tight]{};

                \coordinate (y3) at (4,0);
                \draw[fill=black] (y3) circle (5pt);
                \node at (y3)[label= left:\small $y_3$, tight]{};

                \coordinate (y8) at (6,2);
                \draw[fill=black] (y8) circle (5pt);
                \node at (y8)[label= above:\small $y_8$, tight]{};

                \draw (x1) -- (x7);
                \draw (x1) -- (x3);
                \draw (x4) -- (x6);
                \draw (x2) -- (x8);
                \draw (x6) -- (y3);
                \draw (x8) -- (y8);
  \draw[rounded corners] (-1.1, 7) rectangle (3.1,3.5);
  \draw[rounded corners] (.9,5) rectangle (5.1,1.3);
			\end{tikzpicture}
    \caption{The ringed sets correspond to cocircuits.}
    \label{fig:1245-3}
\end{figure}
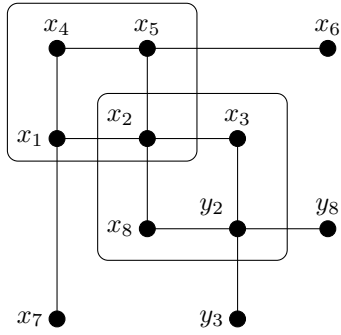

\begin{sublemma}\label{diff}
    $x_1,x_2,\dots,x_8,y_2,y_3,$ and $y_8$ are distinct.
\end{sublemma}

Let $Z = \{x_1,x_2,\dots,x_8,y_2,y_3,y_8\}$. Now $\{x_5,x_6\}$ or $\{x_4,x_6\}$ is in a $4$-cocircuit $S^*_6$ of $M$. By Lemma~\ref{C1C2_intersection_not_2}, $S^*_6$ avoids $\{x_1,x_2,x_3\}$. Thus $S^*_6$ is $\{x_5,x_6,x_8,y_8\}$ or $\{x_4,x_6,x_7,\beta_6\}$ for some element $\beta_6$ not in $Z$. By symmetry, $M$ has a $4$-cocircuit $S^*_7$ containing $x_7$ and $x_1$ or $x_4$. Then $S^*_7$ is $\{x_1,x_3,x_7,y_3\}$ or $\{x_4,x_6,x_7,\beta_7\}$ for some element $\beta_7$ not in $Z$. By symmetry, $M$ has $4$-cocircuits $S^*_3$ and $S^*_8$ where $S^*_3$ contains $\{y_2,y_3\}$ or $\{x_3,y_3\}$, while $S^*_8$ contains $\{y_2,y_8\}$ or $\{x_8,y_8\}$. Then $S_3^*$ is $\{y_2,y_3,y_8,\beta_3\}$ or $\{x_1,x_3,x_7,y_3\}$, and $S^*_8$ is $\{y_2,y_3,y_8,\beta_8\}$ or $\{x_5,x_6,x_8,y_8\}$ where neither $\beta_3$ nor $\beta_8$ is in $Z$.

\begin{sublemma}\label{both}
    If $M$ has both $\{x_5,x_6,x_8,y_8\}$ and $\{x_1,x_3,x_7,y_3\}$ as cocircuits, then $M \cong M(K_{2,2,2})$.
\end{sublemma}

Observe that as $M$ has $\{x_5,x_6,x_8,y_8\}$ and $\{x_1,x_3,x_7,y_3\}$ as cocircuits, $\lambda(Z) \leq 5 + (11-4) - 11 = 1$, so $|E(M) - Z| \leq 1$. If $r(M) = 4$, then deleting the cocircuits $\{x_1,x_2,x_4,x_5\}$ and $\{x_2,x_3,x_8,y_2\}$ from $E(M)$ gives a rank-$2$ flat that contains $\{x_6,x_7,y_3,y_8\}$, a contradiction. Therefore $r(M)\geq~5$. As $r(\{x_1,x_2,x_3,x_4,x_5,x_6,x_7,x_8\})= 4$, to avoid $M$ having a triad, we must have that $|E(M) - Z| = 1$ and $r(M) = 5$. Let $E(M) - Z= \{\delta\}$. Deleting the union of the three cocircuits $\{x_1,x_2,x_4,x_5\}, \{x_2,x_3,x_8,y_2\}$, and $\{x_5,x_6,x_8,y_8\}$ from $E(M)$ leaves $\{x_7,y_3,\delta\}$, so this set is a triangle of $M$. By symmetry, $\{x_6,y_8,\delta\}$ is a triangle of $M$. It follows by Lemma~\ref{tiedown} that $M \cong M(K_{2,2,2})$.

\begin{sublemma}\label{notboth}
      $\{x_5,x_6,x_8,y_8\}$ or $\{x_1,x_3,x_7,y_3\}$ is a cocircuit, or $M \cong M(K_{2,2,2})$. 
\end{sublemma}

 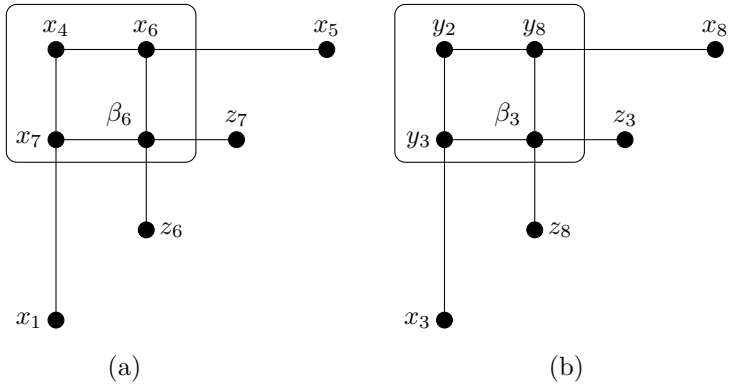
\begin{figure}[b]
      \centering
          \begin{subfigure}{.25\textwidth}
    
  \begin{tikzpicture}[tight/.style={inner sep=1pt}, loose/.style={inner sep=.7em}, scale=0.60]

            	\coordinate (x7) at (0,0);
                \draw[fill=black] (x7) circle (5pt);
                \node at (x7)[label=left:\small $x_1$, tight]{};
                
				\coordinate (x4) at (0,4);
	            \draw[fill=black] (x4) circle (5pt);
                \node at (x4)[label=left:\small $x_{7}$, tight]{};
                
                \coordinate (x1) at (0,6);
                \draw[fill=black] (x1) circle (5pt);
                \node at (x1)[label=above :\small $x_{4}$, tight]{};     
                
				\coordinate (x2) at (2,6);
                \draw[fill=black] (x2) circle (5pt);
                \node at (x2)[label=above:\small $x_{6}$, tight]{};
                
				\coordinate (x3) at (6,6);
                \draw[fill=black] (x3) circle (5pt);
                \node at (x3)[label=above :\small $x_5$, tight]{};
            
                \coordinate (x5) at (2,4);
                \draw[fill=black] (x5) circle (5pt);
                \node at (x5)[label=above left:\small $\beta_6$, tight]{};     
                
                \coordinate (x6) at (4,4);
                \draw[fill=black] (x6) circle (5pt);
                \node at (x6)[label=above :\small $z_7$, tight]{};         
                
                \coordinate (x8) at (2,2);
                \draw[fill=black] (x8) circle (5pt);
                \node at (x8)[label= right:\small $z_6$, tight]{};

                \draw (x1) -- (x7);
                \draw (x1) -- (x3);
                \draw (x4) -- (x6);
                \draw (x2) -- (x8);
  \draw[rounded corners] (-1.1, 7) rectangle (3.1,3.5);

			\end{tikzpicture}
\caption{}
 \label{fig:63:a}
\end{subfigure}%
\hspace{1cm}
\begin{subfigure}{.5\textwidth}
  \centering
  \begin{tikzpicture}[tight/.style={inner sep=1pt}, loose/.style={inner sep=.7em}, scale=0.60]

            	\coordinate (x7) at (0,0);
                \draw[fill=black] (x7) circle (5pt);
                \node at (x7)[label=left:\small $x_3$, tight]{};
                
				\coordinate (x4) at (0,4);
	            \draw[fill=black] (x4) circle (5pt);
                \node at (x4)[label=left:\small $y_{3}$, tight]{};
                
                \coordinate (x1) at (0,6);
                \draw[fill=black] (x1) circle (5pt);
                \node at (x1)[label=above :\small $y_{2}$, tight]{};  
                
				\coordinate (x2) at (2,6);
                \draw[fill=black] (x2) circle (5pt);
                \node at (x2)[label=above:\small $y_{8}$, tight]{};
                
				\coordinate (x3) at (6,6);
                \draw[fill=black] (x3) circle (5pt);
                \node at (x3)[label=above :\small $x_8$, tight]{};
            
                \coordinate (x5) at (2,4);
                \draw[fill=black] (x5) circle (5pt);
                \node at (x5)[label=above left:\small $\beta_3$, tight]{};
                
                \coordinate (x6) at (4,4);
                \draw[fill=black] (x6) circle (5pt);
                \node at (x6)[label=above :\small $z_3$, tight]{};
                
                \coordinate (x8) at (2,2);
                \draw[fill=black] (x8) circle (5pt);
                \node at (x8)[label= right:\small $z_8$, tight]{};

                \draw (x1) -- (x7);
                \draw (x1) -- (x3);
                \draw (x4) -- (x6);
                \draw (x2) -- (x8);
  \draw[rounded corners] (-1.1, 7) rectangle (3.1,3.5);
  
			\end{tikzpicture}
            \caption{}
            \label{fig:63:b}
\end{subfigure}
      \caption{The symmetry between two cases in Lemma \ref{last}.}
      \label{fig:63}
  \end{figure}

Assume that neither $\{x_5,x_6,x_8,y_8\}$ nor $\{x_1,x_3,x_7,y_3\}$ is a cocircuit of $M$. Then both $\{x_4,x_6,x_7,\beta_6\}$ and $\{y_2,y_3,y_8,\beta_3\}$ are cocircuits of $M$ where neither $\beta_6$ nor $\beta_3$ is in $Z$, although $\beta_6$ and $\beta_3$ may be equal. By Lemma~\ref{new3}, $M$ has distinct triangles containing $\{x_6,\beta_6\}$ and $\{x_7,\beta_6\}$, and $M$ has distinct triangles containing $\{y_3,\beta_3\}$ and $\{y_8,\beta_3\}$.

Suppose $\beta_6 = \beta _3$. Then $M$ has $\{x_6,y_3,\beta_6\}$ and $\{x_7,y_8,\beta_6\}$ as triangles, or $M$ has $\{x_6,y_8,\beta_6\}$ and $\{x_7,y_3,\beta_6\}$ as triangles. Then $$\lambda(Z \cup \beta_6) \leq 5 + (12-4) - 12 = 1.$$ Thus $|E(M) - (Z\cup \beta_6\}| \leq 1$. If $|E(M) - (Z\cup \beta_6)| = 1$, let $\nu $ be the element in $E(M) - (Z \cup \beta_6)$. Then $\nu$ is in a triangle of $M$. But the other elements of these triangles are in $Z \cup \beta_6$ and each element of this set is already in two triangles. Thus we have a contradiction to Lemma~\ref{NoThreeTriangles}. We conclude $E(M) = Z \cup \beta_6$. Moreover, when $M$ has $\{x_6,y_8,\beta_6\}$ and $\{x_7,y_3,\beta_6\}$ as triangles, by Lemma~\ref{tiedown}, $M \cong M(K_{2,2,2})$ because $r(M) = 5$. Finally, when $M$ has $\{x_6,y_3,\beta_6\}$ and $\{x_7,y_8,\beta_6\}$ as triangles, the matroid $M / \beta_6\backslash y_3,y_8$ has rank $4$ and has six triangles as in Figure \ref{fig:k33}. Thus, by Lemma \ref{lem33}, this matroid is isomorphic to $M^*(K_{3,3})$, a contradiction. We may now assume that $\beta_6 \not= \beta_3$.

  As $\{x_4,x_6,x_7,\beta_6\}$ is a cocircuit of $M$, by Lemma~\ref{new3}, $\{x_7,z_7,\beta_6\}$ and $\{x_6,z_6,\beta_6\}$ are triangles for some elements $z_6$ and $z_7$. By Lemma~\ref{new4}, $x_1,x_4,x_5,x_6, x_7,\beta_6,z_6,$ and $z_7$ are distinct. Moreover, $M$ has a $4$-cocircuit $D^*_6$ that contains $\{x_5,x_6,z_6\}$ or $\{x_1,x_7,z_7\}$. By symmetry, as $\{y_2,y_3,y_8,\beta_3\}$ is a cocircuit, $M$ has triangles  $\{y_3,z_3,\beta_3\}$ and $\{y_8,z_8,\beta_3\}$ for some elements $z_3$ and $z_8$, where the elements $x_3,x_8,y_2,y_3,y_8,z_3,z_8,$ and $\beta_3$ are distinct. Moreover, $M$ has a $4$-cocircuit $D^*_3$ that contains $\{x_8,y_8,z_8\}$ or $\{x_3,y_3,z_3\}$.

 Suppose $\{x_5,x_6,z_6\}\subseteq D_6^*$. The triangle $\{x_2,x_5,x_8\}$ implies that $x_2$ or $x_8$ is in $D^*_6$. Now $x_2 \not = z_6$ otherwise $x_2$ is in three triangles, a contradiction to Lemma \ref{NoThreeTriangles}. If $x_2 \in D^*_6$, then $x_1$ or $x_3$ is in $D^*_6$. Thus $z_6 \in \{x_1,x_3\}$, so $z_6$ is in three triangles, a contradiction to Lemma \ref{4co}. Thus $x_2 \not\in D^*_6$, so $x_8 \in D^*_6$. Now $x_8 \not= z_6$, otherwise $x_8$ is in three triangles. By orthogonality, $y_2$ or $y_8$ is in $D^*_6$. Thus $z_6 \in \{y_2,y_8\}$. If $z_6 = y_2$, then $z_6$ is in the triangles $\{x_3,y_3,z_6\}, \{x_8,y_8,z_6\}$, and $\{x_6,z_6,\beta_6\}$, a contradiction. Thus $z_6 = y_8$, so $z_6$ is in the triangles $\{x_8,y_2,z_6\}, \{z_6,z_8,\beta_3\},$ and $\{x_6,z_6,\beta_6\}$, a contradiction. We deduce that $\{x_5,x_6,z_6\} \not\subseteq D^*_6$. By the symmetry shown in Figure \ref{fig:63}, $\{x_1,x_7,z_7\} \not\subseteq D^*_6$. We conclude that \ref{notboth} holds.

  We may now assume that $M$ has as cocircuits

  \begin{enumerate}[label=(\roman*)]
      \item both $\{x_5,x_6,x_8,y_8\}$ and $\{y_2,y_3,y_8,\beta_3\}$, or
      \item both $\{x_1,x_3,x_7,y_3\}$ and $\{x_4,x_6,x_7,\beta_6\}$.
  \end{enumerate}
  To see that these two cases are symmetric, recall that we began knowing that $M$ has $\{x_1,x_2,x_4,x_5\}$ as a $4$-cocircuit. Then we assumed, by symmetry, that $M$ has $\{x_2,x_3,x_8,y_2\}$ as a cocircuit. The two cases noted above can be represented as in Figure \ref{fig:ab}.
  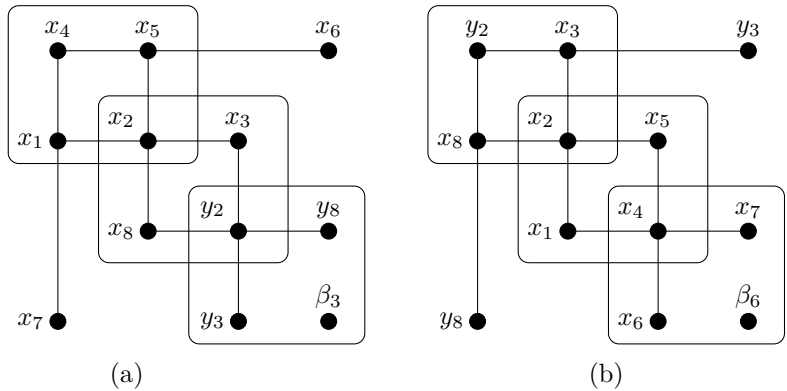
\begin{figure}[b]
      \centering
          \begin{subfigure}{.25\textwidth}
    
    \begin{tikzpicture}[tight/.style={inner sep=1pt}, loose/.style={inner sep=.7em}, scale=0.60]

            	\coordinate (x7) at (0,0);
                \draw[fill=black] (x7) circle (5pt);
                \node at (x7)[label=left:\small $x_7$, tight]{};
                
				\coordinate (x4) at (0,4);
	            \draw[fill=black] (x4) circle (5pt);
                \node at (x4)[label=left:\small $x_{1}$, tight]{};
                
                \coordinate (x1) at (0,6);
                \draw[fill=black] (x1) circle (5pt);
                \node at (x1)[label=above :\small $x_{4}$, tight]{};   
                
				\coordinate (x2) at (2,6);
                \draw[fill=black] (x2) circle (5pt);
                \node at (x2)[label=above:\small $x_{5}$, tight]{};
                
				\coordinate (x3) at (6,6);
                \draw[fill=black] (x3) circle (5pt);
                \node at (x3)[label=above :\small $x_6$, tight]{};
            
                \coordinate (x5) at (2,4);
                \draw[fill=black] (x5) circle (5pt);
                \node at (x5)[label=above left:\small $x_2$, tight]{};            
                \coordinate (x6) at (4,4);
                \draw[fill=black] (x6) circle (5pt);
                \node at (x6)[label=above :\small $x_3$, tight]{};

                \coordinate (y2) at (4,2);
                \draw[fill=black] (y2) circle (5pt);
                \node at (y2)[label= above left:\small $y_2$, tight]{};             
                
                \coordinate (x8) at (2,2);
                \draw[fill=black] (x8) circle (5pt);
                \node at (x8)[label= left:\small $x_8$, tight]{};

                \coordinate (y3) at (4,0);
                \draw[fill=black] (y3) circle (5pt);
                \node at (y3)[label= left:\small $y_3$, tight]{};

                \coordinate (y8) at (6,2);
                \draw[fill=black] (y8) circle (5pt);
                \node at (y8)[label= above:\small $y_8$, tight]{};

                 \coordinate (b8) at (6,0);
                \draw[fill=black] (b8) circle (5pt);
                \node at (b8)[label= above:\small $\beta_3$, tight]{};               

                \draw (x1) -- (x7);
                \draw (x1) -- (x3);
                \draw (x4) -- (x6);
                \draw (x2) -- (x8);
                \draw (x6) -- (y3);
                \draw (x8) -- (y8);
  \draw[rounded corners] (-1.1, 7) rectangle (3.1,3.5);
  \draw[rounded corners] (.9,5) rectangle (5.1,1.3);
  \draw[rounded corners] (2.9,3) rectangle (6.8,-0.5);
			\end{tikzpicture}
            \caption{}
            \label{fig:ab:a}
\end{subfigure}%
\hspace{1.5cm}
\begin{subfigure}{.5\textwidth}
  \centering
    \begin{tikzpicture}[scale=0.60, tight/.style={inner sep=1pt}, loose/.style={inner sep=.7em}]

            	\coordinate (x7) at (0,0);
                \draw[fill=black] (x7) circle (5pt);
                \node at (x7)[label=left:\small $y_8$, tight]{};
                
				\coordinate (x4) at (0,4);
	            \draw[fill=black] (x4) circle (5pt);
                \node at (x4)[label=left:\small $x_{8}$, tight]{};
                
                \coordinate (x1) at (0,6);
                \draw[fill=black] (x1) circle (5pt);
                \node at (x1)[label=above :\small $y_{2}$, tight]{};   
                
				\coordinate (x2) at (2,6);
                \draw[fill=black] (x2) circle (5pt);
                \node at (x2)[label=above:\small $x_{3}$, tight]{};
                
				\coordinate (x3) at (6,6);
                \draw[fill=black] (x3) circle (5pt);
                \node at (x3)[label=above :\small $y_3$, tight]{};
            
                \coordinate (x5) at (2,4);
                \draw[fill=black] (x5) circle (5pt);
                \node at (x5)[label=above left:\small $x_2$, tight]{};            
                \coordinate (x6) at (4,4);
                \draw[fill=black] (x6) circle (5pt);
                \node at (x6)[label=above :\small $x_5$, tight]{};

                \coordinate (y2) at (4,2);
                \draw[fill=black] (y2) circle (5pt);
                \node at (y2)[label= above left:\small $x_4$, tight]{};             
                
                \coordinate (x8) at (2,2);
                \draw[fill=black] (x8) circle (5pt);
                \node at (x8)[label= left:\small $x_1$, tight]{};

                \coordinate (y3) at (4,0);
                \draw[fill=black] (y3) circle (5pt);
                \node at (y3)[label= left:\small $x_6$, tight]{};

                \coordinate (y8) at (6,2);
                \draw[fill=black] (y8) circle (5pt);
                \node at (y8)[label= above:\small $x_7$, tight]{};

                \coordinate (b6) at (6,0);
                \draw[fill=black] (b6) circle (5pt);
                \node at (b6)[label= above:\small $\beta_6$, tight]{};    

                \draw (x1) -- (x7);
                \draw (x1) -- (x3);
                \draw (x4) -- (x6);
                \draw (x2) -- (x8);
                \draw (x6) -- (y3);
                \draw (x8) -- (y8);
  \draw[rounded corners] (-1.1, 7) rectangle (3.1,3.5);
  \draw[rounded corners] (.9,5) rectangle (5.1,1.3);
    \draw[rounded corners] (2.9,3) rectangle (6.8,-0.5);
			\end{tikzpicture}
            \caption{}
            \label{fig:ab:b}
\end{subfigure}
      \caption{The cocircuits include the ringed sets along with $\{x_5,x_6,x_8,y_8\}$ or  $\{x_1,x_3,x_7,y_3\}$, respectively.}
      \label{fig:ab}
  \end{figure}

By symmetry, we may assume that Figure \ref{fig:ab}(a) holds. By \ref{diff}, \linebreak$x_1,x_2,\dots,x_8,y_2,y_3,$ and $y_8$ are distinct. Moreover, we noted when $\beta_3$ was introduced that  it is not equal to any of these eleven elements. By Lemmas \ref{new3} and \ref{new4}, $M$ has triangles $\{y_3,z_3,\beta_3\}$ and $\{y_8,z_8,\beta_3\}$ for some elements $z_3$ and $z_8$, where $x_3,x_8,y_2,y_3,y_8,z_3,z_8,$ and $\beta_3$ are distinct. Moreover, $M$ has a $4$-cocircuit $D^*$ that contains $\{x_8,y_8,z_8\}$ or $\{x_3,y_3,z_3\}$. If $\{x_8,y_8,z_8\} \subseteq D^*$, then $x_2$ or $x_5$ is in $D^*$. If $x_2 \in D^*$, then $z_8 \in \{x_1,x_3\}$, so $z_8$ is in at least three triangles, a contradiction to Lemma \ref{NoThreeTriangles}. Thus $x_5 \in D^*$, so $z_8 \in \{x_4,x_6\}$. To avoid having $z_8$ in more than two triangles, we must have that $z_8 = x_6$, so $D^* = \{x_5,x_6,x_8,y_8\}$. Similarly, if $\{x_3,y_3,z_3\} \subseteq D^*$, then $z_3 = x_7$ and  $D^*=\{x_1,x_3,x_7,y_3\}$. 

Recall that $Z = \{x_1,x_2,\dots,x_8,y_2,y_3,y_8\}$ and $\beta_3 $ is not in $ Z$. Assume that $\{x_8,y_8,z_8\}\ \subseteq D^*.$ Then $z_8 = x_6$ and $D^* = \{x_5,x_6,x_8,y_8\}$. Thus $\lambda((Z - x_7) \cup\beta_3) \leq 5 + (11-4) - 11 = 1$. Therefore $|E(M) - ((Z - x_7) \cup \beta_3)| \leq 1$. Thus $E(M) = Z \cup \beta_3$, so $z_3 \in Z \cup \beta_3$. As each element of $Z \cup \beta_3$ except $x_7$ is in two triangles, we deduce that $z_3 = x_7$. Then, by Lemma~\ref{tiedown}, we get the contradiction that $M \cong M(K_{2,2,2})$ when $x_6 = z_8$. We may now assume that $\{x_3,y_3,z_3\} \subseteq D^*$. Then $z_3 = x_7$ and $D^* = \{x_1,x_3,x_7,y_3\}$. Thus $\lambda((Z - x_6) \cup \beta_3) \leq 1$, so $E(M) = Z \cup \beta_3$ and $z_8 \in Z \cup \beta_3$. By symmetry with the previous case, $z_8 = x_6$ and so $M \cong M(K_{2,2,2})$ a contradiction.
\end{proof}

\section{Proof of the Main Theorem}\label{Sec:MainProof}
In this section, we prove Theorem
\ref{maintheorem:ternary} and then use that to prove Theorem~\ref{mainTheorem}. 

\begin{proof}[Proof of Theorem \ref{maintheorem:ternary}]
    Assume that every cocircuit of $M$ has size at least four, but $M$ does not have $F_7^-,P_7,  M^*(K_{3,3}),Q_9, M(K_5), M(K_{2,2,2}),$ or $H_{12}$ as a minor. If $M$ is not $3$-connected, then, by Lemma \ref{MainTheorem:2conn}, $M \cong H_{12}$. Thus we may assume that $M$ is $3$-connected. By Lemmas \ref{4CoCctNoTriangle} and \ref{4CocctInd}, we may assume that every $4$-cocircuit of $M$ is independent. By Lemma \ref{NoThreeTriangles}, every element of $M$ is in at most two triangles. If every element of $M$ is in at most one triangle, then, as every element of $M$ is in a triangle, $E(M)$ is a disjoint union of triangles. This contradicts  Lemmas \ref{3tri} and \ref{ring}.
    
    We now know that $M$ has an element $x_1$ that is in two triangles $\{x_1,x_2,x_3\}$ and $\{x_1,x_4,x_7\}$. By Lemma \ref{minCond}, there is a $4$-cocircuit $C^*$ containing $x_1$. As $C^*$ is independent, we may assume that  $C^* = \{x_1,x_2,x_4,x_5\}$ for a new element $x_5$. By Lemma \ref{no222}, there is not a triangle containing $\{x_2,x_4\}$. Now, by Lemma \ref{minCond}, $x_5$ is in a triangle.  Then either there is a triangle containing $\{x_2,x_5\}$ but not $\{x_4,x_5\}$, a triangle containing $\{x_4,x_5\}$ but not $\{x_2,x_5\}$, or there are two triangles, one containing $\{x_2,x_5\}$ and the other containing $\{x_4,x_5\}$. The first two cases are symmetric, so we may assume there is a triangle containing $\{x_2,x_5\}$ but not $\{x_4,x_5\}$. Then, by Lemma~\ref{new3}, $M$ is isomorphic to $P_7, M^*(K_{3,3}), Q_9$, or $M(K_{2,2,2})$, a contradiction. Therefore there is a triangle containing $\{x_2,x_5\}$ and a triangle containing $\{x_4,x_5\}$. As $M$ has no $4$-point lines, these two triangles are distinct. Therefore, by Lemma \ref{last}, $M$ is isomorphic to $P_7, M^*(K_{3,3}), Q_9, M(K_5), M(K_{2,2,2}),$ or $H_{12}$, a contradiction. This contradiction completes the theorem.
\end{proof}

    \begin{proof}[Proof of Theorem \ref{mainTheorem}]
     Assume that every cocircuit of $M$ has size at least four, but $M$ does not have $U_{2,5}, F_7,F_7^-, P_7,M^*(K_{3,3}), Q_9, M(K_5), M(K_{2,2,2})$, or $H_{12}$ as a minor. We first assume that $M$ is not $3$-connected. Then, by Lemma~\ref{2Sum2Mat}, $M = M_1 \oplus_2 M_2$ for some $3$-connected matroids $M_1$ and $M_2$ each of which has rank at least three. Assume that some $M_i$ is not ternary. If $r^*(M_i) = 2$, then $M$ has a cocircuit of size less than four, a contradiction. Thus we may assume that the rank and corank of each $M_i$ are at least three. Hence, by Theorem~\ref{ternarySpliiter}, as $M$ does not have $U_{2,5}$ or $F_7$ as a minor, $M_i \cong F_7^*$. Then $M$ has a triad, a contradiction. We deduce that both $M_1$ and $M_2$ are ternary. Thus, by Lemma~\ref{MainTheorem:2conn}, $M$ has an $H_{12}$-minor, a contradiction.
     
     We now know that $M$ is $3$-connected. Then, by Theorem~\ref{ternarySpliiter}, as $M$ does not have $F_7$ or $U_{2,5}$ as a minor and  $M \not\cong F_7^*$, we deduce that $M$ is ternary. Therefore, by Theorem \ref{maintheorem:ternary}, the theorem holds.
    \end{proof}

\end{document}